\documentclass[epj]{svjour}
\usepackage{graphics}
\usepackage{amscd}                                            
\usepackage[arrow,matrix]{xy}
\usepackage{graphicx}
\usepackage{comment}
\usepackage[breaklinks,colorlinks]{hyperref}
\usepackage{breakcites}
\usepackage{amsmath, latexsym, amssymb}
\newcommand{\R}{{\mathbb R}}

\newcommand{\E}{{\mathbb E}}

\newcommand{\N}{{\mathbb N}}

\newcommand{\Aa}{{\mathcal A}}
\newcommand{\Bb}{{\mathcal B}}
\newcommand{\Cc}{{\mathcal C}}    
\newcommand{\Dd}{{\mathcal D}}

\newcommand{\Ff}{{\mathcal F}}
\newcommand{\Hh}{{\mathcal H}}
\newcommand{\Kk}{{\mathcal K}}

\newcommand{\Ll}{{\mathcal L}}    
\newcommand{\Mm}{{\mathcal M}}    

\newcommand{\Pp}{{\mathcal P}}

\newcommand{\Ss}{{\mathcal S}}

\newcommand{\Uu}{{\mathcal U}}

\newcommand{\Xx}{{\mathcal X}}
\newcommand{\Yy}{{\mathcal Y}}
\newcommand{\Zz}{{\mathcal Z}}

\newcommand{\Om}{{\Omega}}
\newcommand{\om}{{\omega}}
\newcommand{\eps}{{\varepsilon}}

\newcommand{\pb}{{\mathbf p}}

\renewcommand{\a}{{\mathfrak a}}

\def\NABLA#1{{\mathop{\nabla\kern-.5ex\lower1ex\hbox{$#1$}}}}
\def\Nabla#1{\nabla\kern-.5ex{}_#1}

\newcommand{\la}{\langle}
\newcommand{\ra}{\rangle}

\usepackage[table]{xcolor}

\begin{document}
\title{Probabilistic morphisms and Bayesian nonparametrics}

\author{J\"urgen Jost\inst{1} \and  H\^ong V\^an L\^e  \inst{2}\footnote{Research  of HVL was supported  by  GA\v CR-project 18-01953J and	 RVO: 67985840}
	\and Tat Dat  Tran \inst{1, 3}
}                     
%
%
\institute{Max-Planck-Institute for Mathematics in the Sciences,
	Inselstrasse 22, 04103 Leipzig, Germany \and Institute  of Mathematics of the Czech Academey of Sciences,
	Zitna 25, 11567  Praha 1, Czech Republic \and Mathematisches Institut, Universit\"at Leipzig, Augustusplatz 10, 04109 Leipzig, Germany}
\date{Received: date / Revised version: date}
%
\abstract{
In this paper  we  develop  a  functorial language   of
probabilistic morphisms   and  apply  it to some basic  problems  in Bayesian
nonparametrics.
First  we extend and  unify   the Kleisli category of probabilistic  morphisms proposed by  Lawvere and  Giry  with the  category of   statistical models   proposed by Chentsov  and Morse-Sacksteder. Then we introduce  the notion of  a Bayesian statistical model  that formalizes  the notion  of a parameter space with a given prior distribution in Bayesian  statistics.  {We revisit  the   existence  of a posterior   distribution, using  probabilistic morphisms}. In particular, we give an explicit   formula for posterior  distributions of  the  Bayesian  statistical model,  assuming  that  the underlying parameter space
is a  Souslin space  and   the   sample space     is a  subset  in a complete  connected  finite dimensional  Riemannian  manifold.  Then we   give  a new  proof  of  the existence   of  Dirichlet
measures over      any measurable  space   using  a  functorial  property  of the   Dirichlet map constructed by Sethuraman. 
}

\maketitle
\section{Introduction}\label{sec:intr}

Statistical inference is a basic tool for essentially all sciences when one wants to detect some structure behind observed data from some measurable sample space $\Xx$ and make predictions about future observations. Therefore, many useful techniques for making inferences have been developed. Beyond this practical aspect, of course, one also seeks a deeper understanding of the underlying principles. Formal methods for that have been developed in various disciplines, including statistical physics,   in particular,  using tools from category theory. We shall therefore now briefly recall the history of the categorical  approach  to mathematical    statistics, to the extent that it is relevant for our work.

When Lawvere  \cite{Lawvere1962} developed a categorical  approach  to mathematical    statistics, he  introduced the name {\it probabilistic mappings} for Markov kernels.   Lawvere  and   later  Giry \cite{Giry1982} proved  important  functorial  properties
of probabilistic  mappings.  Almost      at  the same time      Morse-Sacksteder \cite{MS1966} and Chentsov \cite{Chentsov1965}   developed  independently  categorical approaches   for  mathematical statistics,  namely the category of statistical models.  
The purpose  of our   paper   is to  unify     their approaches  and develop
the  functorial language  of probabilistic  mappings, renamed  in this paper as probabilistic morphisms  following a suggestion by Juan Pablo  Vigneaux, for  dealing with   important problems  in Bayesian nonparametrics.
In  particular,  we introduce the notion of a Bayesian statistical model  (Definition \ref{def:bstatmodel}) and  express the existence   of  posterior  distributions of the Bayesian statistical model in terms of    the equivalence between statistical models (Proposition \ref{prop:post}), generalizing  Parzygnat's theorem \cite[Theorem 2.1]{Parzygnat2020} for finite  sets.  We   also   prove  a  new  formula  for these posterior
distributions  without  the    dominated measure  condition (Theorem~\ref{thm:bayesm}).

Recall that a Souslin space  is a Hausdorff  space admitting   a  surjective     continuous   mapping    from a  complete  separable  metrizable space, i.e., a  Polish space. In particular, every Polish space   is  a Souslin   space,   and more generally, every    Borel  subset in a Souslin  space is a  Souslin  space \cite[Corolary 6.6.7, p. 22, vol.2]{Bogachev2007}.

\begin{remark}\label{rem:altsoulsin}  1) A Souslin space  need not  be  metrizable, see  e.g. \cite[Example 6.10.78, p. 64, vol. 2]{Bogachev2007}.

2) Every  topological space $\Xx$   has a structure  of   a  measurable  space, also denoted by $\Xx$, whose $\sigma$-algebra is  the Borel $\sigma$-algebra  $\Bb(\Xx)$. In \cite[Definition 16, p.46-III]{DM1978}  Dellacherie and Meyer  called  a measurable space  $(\Xx,\Sigma_{\Xx})$  Souslin,  if  $(\Xx,\Sigma_{\Xx})$ is isomorphic to   a measurable space   $(\Hh, \Bb (\Hh))$  where  $\Hh$ is a Souslin {\it metrizable} space.  Here,    we  call such an $(\Xx,\Sigma_{\Xx})$  a  {\it Souslin measurable space}. In \cite[Theorem 68,  p. 76-III]{DM1978}  Dellacherie and Meyer   proved  that, if $\Xx$  is a Souslin space, then $(\Xx, \Bb (\Xx))$  is a  Souslin measurable space. 

3) Since  any Borel  subset  in a Souslin space  is a Souslin  space,  any standard  Borel space (a measurable space admitting a bijective, bimeasurable correspondence with a Borel subset of a Polish space)  is   a Souslin  measurable  space.
\end{remark}

Our main theorem is then stated as follows:

\begin{theorem}[Main  Theorem]\label{thm:bayesm}
	Suppose   that  $\Xx$ is a  subset of   a connected complete finite dimensional     Riemannian  manifold $(M^n, g)$ with  the induced  metric   generated by $g$ on $M^n$,  and $(\Theta, \mu_\Theta, \pb , \Xx)$ is a Bayesian statistical model, where 
	$\Theta$ is  a Souslin space.  For    any $k \in \N ^+$ and ${\bf x}= (x_1, \cdots,  x_k)\in \Xx^k$
	let $D_r(x_i) $ denote the open  ball of radius $r$ centered at $x_i\in \Xx$.  For $\theta \in \Theta$ we let $p_\theta: = \pb (\theta) \in \Pp (\Xx)$. Let $\mu_\Xx : = \pb_* (\mu_\Theta)\in \Pp (\Xx)$ be the marginal  measure on $\Xx$.
	Then  there  exists a  measurable  subset  $S  \subset \Xx$  of zero $\mu_\Xx$-measure  independent  of ${\bf x}$ and
	a family of   posterior    distributions $\mu_{\Theta|\Xx} (\cdot|{\bf x})$  on $\Theta$ after seeing   data ${\bf x} \in \Xx^k$  
	such that
	\begin{equation}\label{eq:updatem}
	\mu_{\Theta|\Xx}(B|{\bf x}) = \lim_{ r\to 0}  \frac{ \int _B \Pi_{i =1}^k p_\theta (D_r (x_i)) d\mu_\Theta}{\int_\Theta\Pi_{i=1}^k  p_\theta (D_r (x_i)) d\mu_\Theta}
	\end{equation}
	for any $B$ in the Borel $\sigma$-algebra $\Bb(\Theta) = \Sigma_\Theta$, and 
	for    any  ${\bf x}\in (\Xx\setminus S)^k$.
	For ${\bf x}\in  \Xx^k \setminus  (\Xx\setminus S)^k$ we   take  an arbitrary value for  $\mu_{\Theta|\Xx} (\cdot | {\bf  x}) \in \Pp (\Theta) $.  
\end{theorem}

Furthermore,   we show  that  the  singular  set $S$ in Theorem \ref{thm:bayesm}   remains a    zero measure   set  with respect  to the
updated  marginal  measure (Proposition \ref{prop:cons}). Thus  our formula  is consistent  for Bayes   inference
and can be  applied  to    Bayesian nonparametrics  when the  dominated      measure condition    usually       does not hold. See,  e.g. \cite{GV2017},  for a discussion of   problems  in Bayesian nonparametrics when the assumption of the dominated      measure condition is dropped off and therefore we cannot  apply  the classical Bayes  formula.

{Bayes statistics  and  the Bayes  formula  have been   treated
	in  several papers devoted to categorical  approaches  to   probability and statistics, \cite{CJ2019},   \cite{Fritz2020}, \cite{Parzygnat2020}
	but the  related results  in the   cited papers  are  ``synthetic"   as emphasized by Fritz  \cite{Fritz2020},   in the sense that   existence  and  uniqueness theorems therein are formal/abstract,  in  contrast to our Main Theorem, which is   ``analytic", in the sense that  our formula is explicit. To our knowledge,  such an explicit  Bayes formula  did not exist in the classical literature  on  probability and  mathematical statistics, except  for the classical   Bayes formula under the assumption  of  dominating  measures \cite[Theorem 1.31, p. 16]{Schervish1997}, see also (\ref{eq:schervish}).}

In our  paper  we also  provide  many  examples  to demonstrate  the
advantage  of the functorial language of probabilistic morphisms. In particular
we  give  a short  functorial  proof  of the  existence of  Dirichlet measures (also known as Dirichlet processes) over any  measurable space using the Sethuraman construction.    Thus	 our functorial  proof  adds  a
demonstration of the  advantage  of the   functorial language of probabilistic  morphisms. We would   like  to refer the reader to
\cite{Fritz2020}  for   a generalization of  the category of probabilistic  morphisms  and their  applications,  in particular,  Fritz's reconsideration  of several   classical   results  in  mathematical statistics, using  the  categorical approach.  We refer  the reader  to \cite{Le2020} for further applications of probabilistic  morphisms  to
the	category of statistical models  and their  natural geometry.

To give an overview of the remainder  of  our paper let us fix  some notations. For a   measurable space $\Xx$, which  sometimes  we shall denote   by $(\Xx, \Sigma_\Xx)$  if we wish to specify  the underlying $\sigma$-algebra $\Sigma_\Xx$, let $\Ss(\Xx)$, $\Mm(\Xx)$  and $\Pp(\Xx)$  denote  the space of  finite
signed measures  on $\Xx$, the space of  non-negative measures  and the space of  probability measures on $\Xx$, respectively. We also set $\Mm^* (\Xx): = \Mm(\Xx)\setminus \{0\}$  and denote by $L(\Xx)$ \footnote{In \cite[p. 74]{Chentsov1972}  Chentsov  used the notation $L(\Xx, \Sigma_\Xx)$  which  is    equivalent  to our notation,  and in \cite[p. 371, vol. 2]{Bogachev2007}   Bogachev used  the notation $\Ll^\infty_{\Sigma_\Xx}$ instead of our    notation $L(\Xx)$.} 
the space of bounded measurable  functions on $\Xx$. Furthermore, we denote by $1_A$ the indicator  (characteristic)  function of  a measurable  set $A$. For  a $\sigma$-additive
measure $\mu$ we denote by $\mu^*$   its outer measure.

In the second  section    we   extend   the Kleisli category  of  probabilistic morphisms  proposed by  Lawvere
\cite{Lawvere1962} and Giry  \cite{Giry1982}   and unify   it   with  the  category of statistical models   proposed  by Chentsov  \cite{Chentsov1965}, \cite{Chentsov1972}   and by Morse-Sacksteder \cite{MS1966}. In the third  section, using the results in the
second section, we introduce the notion of a Bayesian statistical model that formalizes  the notion
of a parameter space with a prior distribution in Bayesian statistics. Then   we   express  the  existence  of a posterior  distribution  using  functorial language  and   give a 
formula  for the     posterior  distribution without   the dominated measure model  condition. This will depend on precise measure theoretical arguments, including a version of the Besicovich covering lemma for Riemannian manifolds proved in \cite[2.8.9]{Federer1969}, see also \cite{JLT2020}.  In the fourth section we    give  a new  proof  of  the existence of   Dirichlet
measures over      any measurable  space   using  a  functorial  property  of the   Dirichlet map constructed by Sethuraman.  

\section{Probabilistic morphisms and category of statistical models}\label{sec:stoch} 
In this section,  first, we extend   Lawvere's  natural   $\sigma$-algebra on $\Pp(\Xx)$ to  the spaces $\Ss(\Xx)$,  $\Mm(\Xx)$, $\Mm^*(\Xx)$, $L(\Xx)$.
We prove   their important  properties (Propositions \ref{prop:sigmainduced}, \ref{prop:st}, \ref{prop:markov},  \ref{prop:tauv}) that shall  be  needed  in later  sections. 
Then we    extend   related   results due  to Lawvere, Chentsov, Giry  and Morse-Sacksteder (Theorem \ref{thm:cat}, Remark \ref{rem:diag}) concerning the  Kleisli category  of probabilistic mappings and show   that   the concept of   the category of statistical models introduced   by  Chentsov  and Morse-Sacksteder  
fits   into the category  of probabilistic
morphisms  (Proposition \ref{prop:suffs}, Theorem  \ref{thm:sprob}, Remark \ref{rem:suffs}).

\subsection{Natural $\sigma$-algebras on $\Ss(\Xx)$,  $\Mm(\Xx)$,  $\Pp(\Xx)$  and $L(\Xx)$}\label{subs:wtop}

Let $\Xx$ be a measurable space. Since we want to integrate not only on $\Xx$ itself, but also on spaces of measures on $\Xx$, we need to construct $\sigma$-algebras on those spaces. For that purpose,   let $\Ff_s(\Xx)$ denote
the  linear space of simple (step) functions on $\Xx$. Such a function is a finite linear combination of indicator functions of measurable subsets of $\Xx$, and therefore, the space $\Ff_s(\Xx)$ depends on the $\sigma$-algebra $\Sigma_\Xx$. We shall now use this space to construct induced $\sigma$-algebras on our spaces of measures on $\Xx$.  There  is a natural  homomorphism $I: \Ff_s(\Xx) \to \Ss^*(\Xx):= Hom (S(\Xx), \R), \, f \mapsto  I_f$,
defined  by  integration:  $I_f (\mu):= \int_\Xx f d\mu$  for $f \in \Ff_s(\Xx)$ and $\mu \in \Ss(\Xx)$. 
Following Lawvere  \cite{Lawvere1962},  we shall denote by $\Sigma_w$   the smallest $\sigma$-algebra  on $\Ss(\Xx)$  such that $I_f$  is measurable  for all $f\in \Ff_s (\Xx)$.   We also  denote  by $\Sigma_w$  the restriction  of  $\Sigma_w$
to $\Mm (\Xx)$, $\Mm^* (\Xx)$  and  $\Pp (\Xx)$. Since  $I_f : \Pp(\Xx) \to  \R$ is bounded, if  $f$ is bounded,  by the  Lebesgue  dominated convergence theorem, the $\sigma$-algebra $\Sigma_w$  on $\Pp (\Xx)$ is    the smallest $\sigma$-algebra
on  $\Pp (\Xx)$ such that   $I_f$ is measurable for all $f \in L(\Xx)$.

\begin{remark}
	\label{rem:sigmaw}
	Lawvere \cite{Lawvere1962}  and later  authors \cite{GH1989}, \cite{Kallenberg2017}\footnote{Kallenberg in \cite[p.1]{Kallenberg2017} defined  $\Sigma_w$ on the  space   of all locally finite measures  on $\Xx$  as in \cite{Lawvere1962}.}
	defined  the  $\sigma$-algebra  on $\Pp(\Xx)$ as the  smallest      $\sigma$-algebra for which the evaluation  map $ev_A: \Pp \to [0,1], \, \mu \mapsto \mu(A),$ is   measurable  for all $A \in  \Sigma_\Xx$. It is not hard  to see  that  their definition  is  equivalent  to ours,  since  $\{ I_{1_A}|\,  A \in \Sigma_\Xx\}$ generate  the vector space $\Ff_s$.  The latter space leads  directly to the  space $L(\Xx)$. 
\end{remark}

For a   topological  space $\Xx$   we  shall  take as $\Sigma_\Xx$ the natural  Borel  $\sigma$-algebra  $\Bb(\Xx)$, unless  otherwise specified. Let  $C_b(\Xx)\subset L(\Xx)$ be the space 
of bounded  continuous functions  on $\Xx$. We denote by $\tau_w$ 
the smallest topology on  $\Ss(\Xx)$  such that 
for any  $f\in C_b (\Xx)$    the map $I_f: (\Ss(\Xx), \tau_w)\to \R$ is continuous. We also denote by
$\tau_w$ the     restriction  of $\tau_w$  to $\Mm(\Xx)$ and $\Pp(\Xx)$, which  is also called  the  weak topology. 
It is known that $(\Pp(\Xx), \tau_w)$ is separable, metrizable  if and only if  $\Xx$ is  \cite[Theorem 3.1.4, p. 104]{Bogachev2018},
\cite[Theorem 6.2, p.43]{Parthasarathy1967} \footnote{If $\Xx$ is infinite, then $(\Ss(\Xx), \tau_w)$ is non-metrizable \cite[p. 102]{Bogachev2018}  (warning: Bogachev's $\Mm(\Xx)$ is our $\Ss(\Xx)$).}.   If $\Xx$ is separable and metrizable then  the Borel $\sigma$-algebra on $\Pp(\Xx)$ generated by $\tau_w$  coincides with $\Sigma_w$ \cite[Theorem 2.3]{GH1989}. 

\begin{proposition}\label{prop:sigmainduced} 	(1) Assume that     $\Sigma _\Xx$   has a  countable generating      algebra  $\Aa_\Xx$. Then  $\Mm(\Xx)$ is a measurable  subset  of $\Ss(\Xx)$, and $\Pp (\Xx)$ and $ \Mm^* (\Xx)$ are   measurable  subsets  of $\Mm(\Xx)$. 
	
	(2) The  addition  $\a: (\Mm(\Xx) \times \Mm(\Xx),\Sigma_w\otimes \Sigma_w) \to (\Mm(\Xx), \Sigma_w), \:  (\mu, \nu)\mapsto \mu +\nu$, is  a measurable map.
	If $\Xx$ is  a  topological space, then the map  $\a$ is $\tau_w$-continuous, i.e., continuous in the $\tau_w$-topology.	
\end{proposition}

\begin{proof} 1.
	The  $\sigma$-algebra $\Sigma_w$ on $\Ss(\Xx)$  is generated  by subsets  $\la A, B^*\ra: = I^{-1}_{1_A}  (B^*)$  where  $A \in \Sigma_\Xx$  and $B^* \in \Bb (\R)$.
	We have 
	$$\Mm(\Xx) =  \cap _{ A  \in \Aa_\Xx}\la  A, \R_{\ge 0}\ra,$$
	because an element of $\Mm(\Xx)$ has to be nonnegative on every $A  \in \Aa_\Xx$. When  $\Sigma _\Xx$   has a  countable generating      algebra  $\Aa_\Xx$, this is a countable intersection, implying the measurability of $\Mm(\Xx)$.
	And since 
	$$\Mm^* (\Xx) = \Mm (\Xx) \cap  \la   \Xx, \R_{> 0} \ra $$
	and
	$$ \Pp (\Xx) = \Mm (\Xx)\cap  \la \Xx, 1 \ra $$
	we then also obtain  the measurability   of $\Mm^*(\Xx)$ and $\Pp (\Xx)$. This proves (1).

	2. To prove the  measurability  of the map $\a$ it suffices to show that for any
	$f \in \Ff_s(\Xx)$ the composition $I_f \circ \a: \Mm(\Xx)\times \Mm(\Xx) \to \R_{\ge 0}$ is measurable. Using the formula
	$$I_f\circ  \a (\mu, \nu) = I_{1_\Xx}(f\cdot \mu) + I_{1_\Xx}(f \cdot \nu),$$
	we reduce  the measurability  of $I_f \circ \a$  to the   measurability of the map
	$\a: \R_{\ge 0} \times \R_{\ge 0} \to \R_{\ge 0}, (x, y) \mapsto (x+y),$ which is well-known. 
	
	Similarly we prove the continuity  of the  map $\a$, if $\Xx$ is a   topological space. This proves  Proposition \ref{prop:sigmainduced}(2). 
\end{proof}

Since  the inclusions  $\Mm(\Xx) \to \Ss(\Xx)$, $\Mm^* (\Xx) \to \Ss(\Xx)$ and  $\Pp (\Xx) \to \Ss(\Xx)$) are   measurable maps, any measure  on  $\Mm(\Xx)$, on $\Mm^* (\Xx)$  and  on $\Pp(\Xx)$  is  obtained by restricting  a measure  on $\Ss(\Xx)$  to 
$\Mm(\Xx)$, $\Mm^*(\Xx)$ and $\Pp(\Xx)$, respectively.  By  Proposition \ref{prop:sigmainduced}   the restriction of any measure  on $\Ss(\Xx)$ to  $\Mm(\Xx)$, $\Mm^*(\Xx)$  and $\Pp^*(\Xx)$)  is a measure on
$\Mm (\Xx)$,  $\Mm^*(\Xx)$  and $\Pp^*(\Xx)$, respectively  if $\Sigma_\Xx$   is countably generated. 

\subsection{Probabilistic  morphisms and associated  functors}\label{subs:propbm}

\begin{definition}\label{def:probmap}  {\it A probabilistic  morphism}
	from   a measurable space $\Xx$ to a measurable  space $\Yy$  (or an arrow   from $\Xx$ to $\Yy$)  is a
	measurable  mapping  from  $\Xx$ to $ (\Pp (\Yy), \Sigma_w)$.
\end{definition}

\begin{remark}\label{def:kleisli}  Definition \ref{def:probmap}
	agrees  with the definition  of  morphisms  in  the Kleisli category  of the
	Giry probability  monad, see  Theorem \ref{thm:cat} and Remark \ref{rem:diag}  below.  (At this stage   we  have not  defined   the category where   a  probabilistic morphism   is   a morphism).  Furthermore  we     would like  to emphasize   that    a   probabilistic morphism   from $\Xx \leadsto \Yy$  {\it is not}   a  mapping   from $\Xx$ to $\Yy$  as it assigns to $x\in \Xx$ not a single point $y\in \Yy$, but rather a probability distribution on $\Yy$. 
	A probabilistic morphism  is  the same as a Markov kernel, see also Remark \ref{rem:diag}.
\end{remark}

We shall  denote by $\overline{T}: \Xx \to  (\Pp(\Yy), \Sigma_w)$ the measurable  mapping  defining/generating a probabilistic  morphism $T: \Xx \leadsto \Yy$. 
Similarly, for a measurable mapping $\pb: \Xx \to \Pp(\Yy)$  we shall denote by $\underline{\pb}: \Xx \leadsto \Yy$ the generated probabilistic  morphism.   Note that    a   probabilistic  morphism  is  denoted  by a curved  arrow  and  a measurable mapping by a straight arrow.

\begin{example}\label{ex:prob1}
	
	(1) Assume that $\Xx$ is separable and metrizable. Then the  identity mapping $Id_\Pp:(\Pp(\Xx), \tau_ww \to (\Pp (\Xx), \tau_w)$ is  continuous, and hence 
	measurable w.r.t.  the Borel $\sigma$-algebra $\Sigma_w = \Bb(\tau_w)$. Consequently $Id_\Pp
	$  generates    a probabilistic morphism $ev: (\Pp (\Xx), \Bb(\tau_w))  \leadsto  (\Xx, \Bb(\Xx))$  and we write  $\overline {ev} = Id_\Pp$.  Similarly, for any measurable space $\Xx$,   we also  have   a probabilistic morphism  $ev: (\Pp(\Xx), \Sigma_w) \leadsto \Xx$  generated  by     the  measurable mapping $\overline {ev} = Id_\Pp$.

	(2) In   generative models of supervised learning,   we   wish to know  the stochastic  correlation   between   a label $y$ in a sample space $\Yy$  and      an
	input $x$ in  a sample space  $\Xx$ that expresses the uncertain nature of the relationship between an input and its label. 
	Such a stochastic correlation is  expressed  via  a   probability measure  $\mu_{\Xx \times \Yy} \in \Pp(\Xx \times \Yy)$,  assuming usually 	$\Sigma_{\Xx \times \Yy}= \Sigma_\Xx\otimes \Sigma_\Yy$. Denote by $\Pi_\Xx$  and $\Pi_\Yy$  the   natural projections    from $\Xx \times \Yy$ to $\Xx$ and $\Yy$, respectively. One  is  then interested     in  the conditional     probability  
	\begin{equation}\label{eq:xiy}
	\mu_{\Yy|\Xx}(y \in B| x): = \frac{ d(\Pi_\Xx)_* (1_{B \times \Xx}\mu_{\Xx \times \Yy}
		)}{d(\Pi_\Xx)_* \mu_{\Xx \times \Yy}}(x)
	\end{equation}
	for  $B \in \Sigma_\Yy$ and $x\in \Xx$.  Here  the  equality  (\ref{eq:xiy})   should be understood as an  equivalence  class of  functions  in $L^1 (\Xx, (\Pi_\Xx)_* \mu_{\Xx\times \Yy})$.
	For   simplification and practical purposes, we   assume that
	the  conditional   probability  $\mu_{\Yy|\Xx} (\cdot | x)$  is regular, i.e.,   there is  a function  on $\Xx$  that  represents  the equivalence class  of  $\mu_{\Yy|\Xx} (\cdot | x)$, and therefore we   shall    denote  it  also by $\mu_{\Yy|\Xx} (\cdot | x)$ abusing notation,  such that  $\mu_{\Yy|\Xx} (\cdot | x)$
	is a measure  on $\Yy$ for all $x\in \Xx$,   and  the  function $ x\mapsto   \mu_{\Yy|\Xx} (B|x)$ is   measurable   for  every  $B\in \Sigma_\Yy$. In this  case
	the map $\overline T: \Xx \to \Pp(\Yy), x\mapsto  \mu_{\Yy|\Xx} (\cdot|x)$, is a measurable map.  Thus   under   this  regularity condition  the problem  of  supervised  learning with generative models  is   equivalent  to  the  problem of   finding     a  probabilistic  morphism $T: \Xx \leadsto \Yy$  describing  the correlation between   labeled  data  $y$ and  input data $x$.

	(3)	 We would like to present here a non-trivial construction of probabilistic morphisms, which comes from the theory of random mappings \cite{Kifer1986}, \cite{Kifer1988}. Consider a probability space $(\Omega, \mathcal{F},\mathbb{P})$ and a {\it random mapping} $T$, which is a measurable mapping $T: \Omega \times \Xx \to \Yy$ (in case $\Xx = \Yy$, the random mapping $T$ is the source to generate a random dynamical system \cite{Arnold1998}).\\
	Such a random mapping $T: \Omega \times \Xx \to \Yy$ then generates a measurable mapping $\overline{T}: \Xx \to \Pp(\Yy)$, defined by
	\begin{equation}\label{randommap}
	\overline{T}(x)(B) := \mathbb{P} \{\omega \in \Omega: T(\omega,x) \in B\} = \mathbb{P}\{ T^{-1}(\cdot,x)(B)\},\quad \forall B \in \Sigma_\Yy.  
	\end{equation}
	Relation \eqref{randommap} shows that once a probabilistic sample space $(\Omega, \mathcal{F},\mathbb{P})$ is given, any random mapping would generate a probabilistic morphism. However, since probabilistic morphisms  need not be constructed this way, the concept of probabilistic morphisms can be seen as a generalization of random mappings.\\ 
	We are also interested in the inverse problem on the representation of probability measures, which can be formulated as follows: Given a probabilistic morphism
	$T: \Xx \leadsto \Yy $, is there a probabilistic sample space $(\Omega, \mathcal{F},\mathbb{P})$ such that $T$ can be written as a random mapping from $\Omega \times \Xx$ to $\Yy$?  This question appears in several contexts in Ergodic Theory \cite{Kifer1986}, \cite{Kifer1988}. The answer is affirmative if $\Yy$ is a compact, oriented, connected manifold of class $C^{2}$ and $\Xx$ is a $C^0$ manifold with some additional assumptions on the decay of the densities, where one can choose $\Omega := \Yy$ (see details in \cite[Theorem A]{JMPR2019}).
\end{example}

Given  a    probabilistic  morphism   $T:  \Xx \leadsto \Yy$,   we define   a   linear  map
$T^*: L(\Yy) \to L(\Xx)$   as follows
\begin{equation}\label{eq:adjoint1}
T ^*(f)  (x): = I_f (\overline T (x))  = \int _\Yy f d  \overline T (x).
\end{equation}

Formula  (\ref{eq:adjoint1}) coincides  with the  classical formula  (5.1)  in \cite[p. 66]{Chentsov1972} for  the transformation  of  a  bounded measurable  function  $f$ under  a Markov morphism  $T$.\footnote{Chentsov  called   the  induced transformation  $P_*(T)$ (see Remark \ref{rem:l0}(2))  of a probabilistic morphism  $T$  a Markov morphism \cite{Chentsov1965}.}
In particular, if $\kappa: \Xx \to \Yy$ is a measurable mapping,  then   we have $ \kappa^*(f) (x) = f(\kappa(x))$, since  $\overline  \kappa = \delta\circ \kappa$.

Further,   we   define  a linear  map $S_*(T): \Ss(\Xx) \to \Ss(\Yy)$  as follows  \cite[Lemma 5.9, p. 72]{Chentsov1972}
\begin{equation}\label{eq:markov1}
S_*(T) (\mu) (B): = \int_{\Xx}\overline T (x) (B)d\mu (x)
\end{equation}
for any    $\mu \in \Ss (\Xx)$  and  $B \in \Sigma_\Yy$.

\begin{remark}
	\label{rem:l0}  (1) If  $\kappa: \Xx \to \Yy$  is a measurable mapping then
	$\kappa ^* (\Ff_s (\Yy))\subset \Ff_s (\Xx)$.  For  a general probabilistic
	morphism   $T: \Xx \leadsto \Yy$,  we don't have  the inclusion $ T^* (\Ff_s (\Yy))\subset \Ff_s (\Xx)$, since the   cardinality $\#( T^*  (1_B)(x) |\,  x\in \Xx)$ can  be infinite.  For   example, consider $T = ev: \Xx: = \Pp (\Zz) \leadsto \Yy :=\Zz$  for some   measurable space $\Zz$ and  $ \emptyset  \not= B \not = \Zz$. Then  $ev^* (1_B) (\mu)  = \mu (B)$  for any $\mu \in \Xx = \Pp (\Zz)$. Therefore   $\# (ev ^* (1_B) (\mu)|\,  \mu\in  \Pp(\Zz)) = \# (\mu (B)|\mu \in\Pp(\Zz)) = \infty$, if  $\Zz  = [0,1]$ and  $B = [0,1/2]$.
	
	(2)  Lawvere   also  defined    $P_*(T): \Pp (\Xx) \to \Pp(\Yy)$  by  (\ref{eq:markov1})  for
	$\mu \in \Pp(\Xx)$.
	
\end{remark}

\begin{definition}	\label{def:dual} 
	- We denote  by $\Sigma_*$  the  smallest $\sigma$-algebra on  $L(\Xx)$  such that   for any $\mu \in \Pp(\Xx)$ the  evaluation function $ev_\mu : L(\Xx) \to \R, h \mapsto I_h (\mu)$,  is measurable.	
	
	-	We denote  by $\Sigma_{ev}$ the smallest  $\sigma$-algebra on    $L(\Xx)$  such that 
	for all  $x \in  \Xx$  the evaluation map   $ev_x: L(\Xx) \to \R, h \mapsto h(x),$  is measurable.
	
\end{definition} 

\begin{example}\label{sigmas}
	Identifying  $L (\Om_k)=C_b(\Om_k)$  with   $\Ss(\Om_k)$ via the Euclidean  metric,   and  noting that  the $\sigma$-algebra	$\Sigma_w$ on $\Ss(\Om_k)$     is the  usual   Borel $\sigma$-algebra  on 
	$\R^k = \Ss(\Om_k)$, cf. \cite[Theorem 2.3]{GH1989},  it is not hard to see  that  the $\sigma$-algebra	$\Sigma_*$ on $L(\Om_k)= \R^k$  coincides with the Borel $\sigma$-algebra  on $\R^k$.
	\end{example}

\begin{remark}\label{rem:Lxsigma}  (1)  For any $x \in  \Xx$ denote  by  $\delta_x \in \Pp(\Xx)$ the Dirac 
	measure concentrated at $x$.   Giry proved  that the map  $\delta: \Xx \to (\Pp(\Xx), \Sigma_w), x \mapsto \delta_x,$ is a measurable
	mapping\cite{Giry1982}.   Furthermore, for $x\in \Xx$ and $h \in L(\Xx)$   we have $ev_{\delta(x)}(h)= ev_x(h)$.

	(2) The  requirement  that  the evaluation maps  $ev_x$ are measurable   is natural  in the theory of measurable mappings  \cite{Aumann1961}  with applications in stochastic processes \cite{Baudoin2014}. 
\end{remark}

\begin{proposition}\label{prop:st}  For any  measurable space $\Xx$ we have  $\Sigma_{ev} \subset \Sigma_*$. If  $\Xx$ is  a separable metrizable  space then   $\Sigma_{ev} = \Sigma_*$.
\end{proposition}

\begin{proof}
	Let $\Xx$ be a measurable space   and $x \in \Xx$. Since the  restriction  of   the evaluation  mapping $ev_x$  to 
	$L(\Xx)$ is equal to  the evaluation mapping $ev_{\delta (x)}$,  we  obtain  immediately  the first  assertion of Proposition \ref{prop:st}.
	
	
	To prove the  second assertion  it suffices   to show    that  the evaluation  mapping
	$$ ev_{\mu} :(L(\Xx), \Sigma_{ev}) \to \R,\: h \mapsto  I_h(\mu)  $$
	is measurable  for any $\mu \in \Pp(\Xx)$.  	Since  $\Xx$ is a  separable  metrizable space, by \cite[Example 2.2.7, p. 55]{Bogachev2018}, for any  $k \in \N$ there    exist a partition    of $\Xx$  in  countably many Borel parts, $\Xx = \cup_{i =1} ^\infty B_{k,i}$  and  a sequence   of points $x_{k,i} \in B_{k,i}$  such that  the sequence
$$\mu^k = \sum _{ i =1} ^\infty \mu (B_{k,m}) \delta _{x_{k, i}}$$
converges weakly to $\mu$. 
	
Then  for any  $h \in L(\Xx)$  we have
	$$\mu(h) = \lim_{k \to \infty} \mu^k (h).$$
	Since  $ev_{\mu^k}:  (L(\Xx), \Sigma_{ev}) \to \R$ is measurable, it follows that  $ev_\mu:  (L(\Xx), \Sigma_{ev}) \to \R$ is measurable. This completes the proof of Proposition \ref{prop:st}.
	
\end{proof}

\begin{proposition}\label{prop:markov}  Assume that  $T: \Xx  \leadsto \Yy$  is a probabilistic   morphism. 
	
	(1) Then $T$ induces   a bounded linear   map  $S_*(T): \Ss(\Xx) \to \Ss(\Yy)$ w.r.t. the total variation norm $|| \cdot || _{TV}$. The restriction $M_*(T)$  	of $S_*(T)$ to  $\Mm(\Xx)$  and $P_*(T)$  of $S_*(T)$ to  $\Pp (\Xx)$, maps  $\Mm(\Xx)$ to  $\Mm(\Yy)$   and $\Pp (\Xx)$ to
	$\Pp(\Yy)$, respectively. 
	
	(2) If  $T$  is a measurable mapping, then   $S_*(T):  (\Ss(\Xx), \Sigma_w) \to (\Ss(\Yy), \Sigma_w)$ is a  measurable map. It is  $\tau_w$-continuous  if  $T$ is a continuous  map between separable metrizable spaces.  Hence   the maps $M_*(T)$ and $P_*(T)$
	are   measurable  and   they are     $\tau_w$-continuous  
	if  $T$ is a  continuous map  between separable metrizable spaces.
	
	(3)  If $T_1: \Xx \to \Yy$ and $ T_2: \Yy \to \Zz$ are measurable  mappings then  we have
	\begin{equation}\label{eq:fm}
	S_*(T_2 \circ T_1)= S_*(T_2)\circ S_*(T_1), \: P_*(T_2 \circ T_1)= P_*(T_2)\circ P_*(T_1).
	\end{equation}
	
	(4)  The map $T$ also induces measurable  linear  mappings $T^*:  (L(\Yy), \Sigma_*) \to  (L(\Xx), \Sigma_*)$ and $T^*: (L(\Yy), \Sigma_{ev}) \to (L(\Xx), \Sigma_{ev})  $.  
\end{proposition}

\begin{proof} 1. Proposition \ref{prop:markov} (1) is   due  to  Chentsov \cite[Lemma 5.9, p.72]{Chentsov1972}. 
	
	2. Assume that  $T$ is a measurable mapping.
	To prove that  $ S_*(T)$ is a measurable  mapping, it suffices  to show   that
	for any  $f \in \Ff_s (\Yy)$ the composition $I_f \circ  S_*(T) : \Ss (\Xx) \to 
	\R$ is  measurable.  The latter  assertion  follows  from the   identity
	$I_f \circ  S_*(T) = I_{ T^* f}$, taking into account  that $T^*(f) \in \Ff_s (\Xx)$.  In the same    way  we  prove that  $S_*(T)$ is $\tau_w$-continuous, if $T$ is continuous  map between separable metrizable spaces.
	The statement  concerning $M_*(T)$ and  $S_*(T)$  is  a consequence  of the  first two  assertions of  Proposition \ref{prop:markov}(2).
	Note  that  the measurability of $P_*(T)$  has been  first  noticed by Lawvere  \cite{Lawvere1962}. 
	
	3. The first  identity  in (\ref{eq:fm})  is obvious,  the second  one  is a consequence  of the   first one.
	
	4.   The  linearity of  $ T^*$ and the inclusion $ T^* (L (\Yy) )  \subset L (\Xx)$ have been  proved in   \cite[Corollary, p. 66]{Chentsov1972}. Here  we  provide an alternative  shorter   proof of the  inclusion  relation.
	Let $ T ^*$   be defined by (\ref{eq:adjoint1})  and $f \in L(\Yy)$.  Since
	$$\sup _{x\in \Xx}| T^*(f) (x) |\le \sup_{y \in \Yy}|f(y)|$$
	the  function $T^* (f)$ is bounded on $\Xx$.
	
	Next  we shall  show that   
	for any $f \in L (\Yy)$ the map $ T^*(f):\Xx \to \R$  is measurable, i.e., for   any   Borel  set $I \subset \R$  we have $T^*(f)^{-1}(I) \in \Sigma_\Xx$. Note that
	$$T^*(f)^{-1} (I) = \{ x \in \Xx| \, T^* (f)(x) \in I , \text{ i.e., } I_f (\overline T (x)) \in I\} . $$
	Since $I_f: \Pp (\Yy) \to \R$ is measurable, the  set $I_f ^{-1}(I)$  is measurable.
	Since $\overline T: \Xx \to \Pp(\Yy)$ is measurable,  the  set
	$(T ^*(f))^{-1}(I) = \overline T^{-1} (I_f ^{-1} (I))$  is measurable.  Thus 
	$ T^* (L (\Yy) )  \subset L (\Xx)$. 	
	
	Now  we shall show that   $T^*: (L(\Yy), \Sigma_*)\to  (L(\Xx), \Sigma_*)$  is a measurable  map.  It  suffices  to show  that   for any $\mu \in \Pp (\Xx)$ the composition $ev_\mu \circ T^*: L(\Yy) \to \R$  is  a measurable map.  Let $I \subset \R$ be a Borel subset. Then 
	$$ ev_\mu ^{-1} ( T^*) ^{-1}(I) = \{   g \in L(\Yy)|\: \int_\Xx \int _\Yy  g  d \overline T (x)\, d\mu (x)\in I \}.$$
	Setting
	$$ \overline T _\mu : = \int_\Xx  \overline  T (x) d\mu \in \Pp (\Yy)$$ we get
	\begin{equation}
	\label{eq:meas1}
	ev^{-1}_\mu (T^*)^{-1}(I) =  ev _{\overline T _\mu} ^{-1} (I)
	\end{equation}
	which  is a measurable   subset  of $L(\Yy)$. Hence  $T^*: (L(\Yy), \Sigma_*)\to  (L(\Xx), \Sigma_*)$  is a measurable  map. 
	
	From (\ref{eq:meas1})  we also   deduce  that  $T^*: (L(\Yy), \Sigma_{ev})\to  (L(\Xx), \Sigma_{ev})$  is a measurable  map,
	since $ev_x(h) = ev_{\delta(x)}(h)$ for  all  $x\in \Xx$ and $h \in L(\Xx)$.
	This completes  the proof of  Proposition \ref{prop:markov}.
\end{proof}

\begin{proposition}\label{prop:tauv} (1) The projection $\pi_\Xx: \Mm^*(\Xx) \to \Pp (\Xx), \mu \mapsto \mu(\Xx)^{-1}\cdot \mu$, is a measurable  retraction.  If $\Xx$ is a topological space then $\pi_\Xx$ is $\tau_w$-continuous.
	
	(2)  If     $\Sigma_\Xx$  is countably generated, then  we have the following commutative diagram
	$$
	\xymatrix{\Pp^2 (\Xx) \ar[r]^{i_{p,m}} \ar[d]^{P_* (i_{p,m})} & \Mm(\Pp (\Xx))\ar[r]^{i_{m,s}}\ar[d]^{M_*( i_{p,m})} & \Ss(\Pp(\Xx))\ar[d]^{S_*( i_{p,m})}\\
		\Pp(\Mm(\Xx))\ar[r]^{i_{p,m}}\ar[d]^{P_*(i_{m,s})} &  \Mm^2(\Xx) \ar[r]^{i_{m,s}}\ar[d]^{M_*(i_{m,s})} & \Ss(\Mm(\Xx))\ar[d]^{S_*(i_{m,s})}\\
		\Pp(\Ss(\Xx)) \ar[r]^{i_{p,m}} & \Mm(\Ss(\Xx))\ar[r]^{i_{m,s}}&  \Ss^2(\Xx)
	}
	$$
	where all arrows are  measurable embeddings.   If  $\Xx$ is a  separable  metrizable space then the arrows are $\tau_w$-continuous embeddings.
\end{proposition}

\begin{proof} 1. Clearly $\pi_\Xx$ is a retraction. The map $\pi_\Xx$ is measurable, since 
	the  mapping  $I_{1_\Xx}: \Mm^*(\Xx) \to \R_{> 0}, \: \mu \mapsto \mu(\Xx),$ is measurable, and hence  the map $(I_{1_\Xx})_{-1}:\Mm^* (\Xx)\to \R_{> 0}, \: 
	\mu \mapsto  \mu(\Xx)^{-1},$ is also measurable. Similarly we prove
	that $\pi_\Xx$ is continuous in  the $\tau_w$-topology, if $\Xx$, $\Yy$ are topological spaces.
	
	2. The measurability of the  horizontal mappings $i_{p,m}$ and $i_{m,s}$ follows from
	Proposition \ref{prop:sigmainduced}(1).   If $\Xx$ is a separable metrizable   space,  these  maps
	are inclusion  maps, and hence   they are continuous by definition.
	
	The measurability  and the continuity of the vertical mappings $S_*(i_{p,m})$ and
	$S_*{(i_{m,s})}$
	follow  from the  measurability and the continuity of the maps $i_{p,m}$ , $i_{m, s}$, respectively,  and  Proposition \ref{prop:markov}(2).  The measurability  and the continuity of  the vertical mappings $P_*(i_{p,m})$,
	$S_*{(i_{m,s})}$  and  $M_*(i_{p,m})$, $M_*(i_{m,s})$  follow from the  corresponding
	assertion concerning  the functor $P_*$,  the measurability and  the continuity for the horizontal mappings, respectively.
	
	3. The  commutativity  of the diagram  is   obvious. This completes  the  proof of Proposition  \ref{prop:tauv}.
\end{proof}

Let us recall the following  result due to Giry.

\begin{lemma}\label{lem:ev2} ( \cite{Giry1982}) Let $ev_P: \Pp^2(\Xx)\to  \Pp(\Xx)$ be   defined  by
	\begin{equation}\label{eq:evS} 
	ev_P( \nu_\Pp) (A): = \int_{\Pp(\Xx)}I_{ 1_A}(\mu) \,  d\nu_\Pp (\mu)
	\end{equation}
	for  all $\nu _\Pp \in \Pp^2(\Xx)$ and $A \in \Sigma_\Xx$. 
	
	(1)  The  composition  $ev_P \circ  \delta:  \Pp (\Xx) \to \Pp (\Xx)$ 
	is  the identity  map.
	
	(2) Assume  that $\kappa: \Xx \to \Yy$  is a measurable mapping. Then we have the following commutative diagrams 
		$$
	\xymatrix{  	\Pp^2 (\Xx)\ar[r] ^{P_*^2 (\kappa)}  \ar[d]^{ ev_P} &  \Pp^2(\Yy)\ar[d] ^{ev_P}\\
		\Pp(\Xx)\ar[r]^{P_*(\kappa)}& \Pp(\Yy).
	}
	$$
	
	(3)  The   mapping  $ev_P$ is a   measurable mapping.  It  is 
	$\tau_w$-continuous, if  $\Xx$  is  a separable  metrizable  space. 
\end{lemma}

Note that  Giry    considered   the smaller category of Polish spaces   but  his  proof   is also valid for the  category of separable  metrizable   spaces.

\begin{theorem}\label{thm:cat} (1) Probabilistic morphisms  $T$ are morphisms in the category of  measurable  spaces $\Xx$. Furthermore 
	$(\delta, M_*)$ and  $(\delta, P_*)$ are  faithful functors  from the category  of  measurable spaces  whose morphisms are probabilistic  morphisms to the   category of         nonnegative  finite  measure spaces     and    the category of probability measure spaces  respectively, whose  morphisms are measurable mappings.  If $\overline T: \Xx \to \Pp(\Yy)$ is a continuous mapping  between  separable metrizable spaces then   $M_*(T): \Mm (\Xx) \to \Mm(\Yy)$ and  $P_*(T): \Pp (\Xx) \to \Pp(\Yy)$ are $\tau_w$-continuous. 
	
	(2)  If $\nu \ll \mu\in  \Mm^*(\Xx)$ then  $M_*(T)(\nu)  \ll M_*(T) (\mu)$.
\end{theorem}

\begin{proof}  1. The statement of Theorem   \ref{thm:cat}(1)  concerning    the functor  $P_*$ without mentioning its faithfulness is a consequence  of  Giry's theorem  stating  that  the  triple
	$(P_*, \delta, ev_P)$ is a    monad  in the   category  of measurable
	spaces  whose morphisms  are  measurable mappings and    in the category  of Polish  spaces  whose morphisms are    continuous  mappings,  respectively,  see also   the remark  before  Theorem \ref{thm:cat}  and Remark  \ref{rem:diag} below) \cite[Theorem 1, p. 70]{Giry1982} and his    observation  that  the Kleisli category of the monad $(P_*, \delta, ev_P)$ is  the  category  of  measurable  spaces whose morphisms are    probabilistic morphisms. In other words, Giry's theorem  says  that   Theorem
	\ref{thm:cat}  is valid  in the subcategory  of measurable  spaces  whose morphisms are  measurable mappings,  $\delta  $ is  a natural transformation  of the  functor $Id_\Xx$ to the functor $P_*$ on this   subcategory and $ev_P$  is   a natural  transformation  of the functor $P_*^2 $  to  the the functor $P_*$ on  the same subcategory.  Note that the last  assertion  is  equivalent  to  the statement  of Lemma \ref{lem:ev2}(2)
	(the associativity of $ev_P$  follows from the  identity (\ref{eq:assm}) below).
	
	Taking into account  Giry's theorem,  the statement  of  Theorem \ref{thm:cat}   concerning  the functoriality (without faithfullness) of $P_*$  is  a  consequence   of  \cite[Theorem 1, p. 143]{MacLane1994}   on the structure  of  the Kleisli category  of  a monad.  
	Since  the proof  of \cite[Theorem 1, p.143]{MacLane1994}  is  only sketched,   we shall   prove Proposition \ref{prop:funct2} below,  which proves       the functoriality  of     $P_*$. Our  proof  has also a different  strategy than  the one  sketched in \cite[p. 143]{MacLane1994}.  Since  $M_*(T)(c \cdot \mu) = c \cdot  M_* ( T)(\mu)$ for any $c \in \R_{\ge 0}$ and  $\mu \in \Pp (\Xx)$,  the functoriality of   $M_*(T)$   is a consequence  of the   functoriality  of  $P_*(T)$,  and   the measurability of   $M_*(T)$
	is a  consequence  of  measurability of  $P_*(T)$.  We   provide  furthermore   a  new categorical proof for  the  associativity of   the composition of   the Markov kernels (Proposition \ref{prop:ass}).  
	
	\begin{proposition}\label{prop:funct2}   Let $ T_i: \Xx_i \leadsto \Xx_{i+1}$  be   probabilistic morphisms for $i = 2,3$. Then we have
		\begin{equation}\label{eq:pind}
		P_*(T_3 \circ T_2)= P_*(T_3) \circ P_*(T_2).
		\end{equation}
	\end{proposition} 
	\begin{proof} The  proof consists  of  two steps.  In the first  step, using (\ref{eq:markov1})  we shall   prove relations  (\ref{eq:pinduce1}) and (\ref{eq:compose})  below. (If we  want to construct  the Kleisli category of a monad, we  take  these  formulas  as  defining relations for the functor $P_*$  and for  the  composition rule of morphisms in the Kleisli category.)  Let $T$  and  $T_i : \Xx_i \leadsto \Xx_{i+1}$  be  probabilistic morphisms for $i =1, 2$. Then
		
		\begin{equation}\label{eq:pinduce1} 
		P_*(T) =  ev_P \circ P_*(\overline T),
		\end{equation}

		\begin{equation}\label{eq:compose}
		\overline{T_2\circ T_1}  = P_*( T_2) \circ \overline{ T_1}. 
		\end{equation}	 
		
		Note that   (\ref{eq:pinduce1}) is  a consequence  of  the  following  straightforward  computation for any $B \in \Sigma_\Yy$  and any $\mu \in \Pp (\Xx)$, 
		$$ ev _P \circ P_* (\overline T)  (\mu) (B) = \int_{\Pp (\Yy)}I_{1_B} dP_* (\overline T)\mu = \int_{\Xx}\overline T (x) (B)d\mu (x)= P_*(T) (\mu) (B).$$ 
		
		We leave  the reader  to verify   (\ref{eq:compose}), which   has been  verified  by Giry   \cite{Giry1982} as we   mentioned   above,    (we prefer  to  consider  (\ref{eq:compose}) as  a defining relation  and hence  we  did not recall  the known composition rule for  Markov kernels in this paper).
		
		In the second  step we examine  the following  diagram:
		\begin{eqnarray}\label{eq:*diagr2}
		\xymatrix{
			& & &  \Pp ^3 (\Xx_4)\ar[d]^{P_*(ev_P)}\\
			&  & \Pp^2  (\Xx_3)\ar[ur]^{P_* ^2(\overline {T_3})} \ar[r]^{P_* ^2 (T_3)}\ar[d]^{ev_P} & \Pp^2 (\Xx_4)\ar[d]^{ev_P}\\
			& \Pp(\Xx_2)\ar[ur]^{P_*(\overline {T_2})}\ar[r]^{P_* (T_2)} \ar@{~>}[d]^{ev} &
			\Pp(\Xx_3)\ar[ur]^{P_*(\overline {T_3})}\ar[r]^{P_* (T_3)}  \ar@{~>}[d]^{ev} & 
			\Pp(\Xx_4) \ar@{~>}[d]^{ev}\\
			X_1 \ar[ur]^{ \overline {T_1}}\ar@{~>}[r] ^{T_1} & X_2 \ar[ur]^{\overline {T_2}}\ar@{~>}[r] ^{T_2}&  X_3 \ar[ur]^{\overline {T_3}}\ar@{~>}[r] ^{T_3} & \Xx_4,
		}
		\end{eqnarray}
		By  (\ref{eq:pinduce1}), (\ref{eq:compose}), and using  (\ref{eq:fm}), we have 
		\begin{eqnarray}
		P_*(T_3\circ T_2)= ev_P \circ P_* (\overline{T_3 \circ T_2})\nonumber\\
		= ev_P \circ P_* ( P_*(T_3) \circ \overline{T_2})\nonumber\\
		= ev_P  \circ P_* (ev_P) \circ  P^2_* (\overline T_3)\circ P_* (\overline T_2).\label{eq:comm3}
		\end{eqnarray}
		
		\begin{lemma}\label{lem:funct2} (\cite[p. 71]{Giry1982})  We have   the following  identity
			\begin{equation}\label{eq:assm}
			ev_P\circ P_*(ev_P)  = ev_P \circ ev_P.
			\end{equation}
		\end{lemma}
		From Lemma \ref{lem:ev2} (3)  we  obtain immediately  the following identity
		\begin{equation}\label{eq:commse}
		ev_P\circ  P^2_*   (\overline T)= P_* (\overline T)\circ  ev_P.
		\end{equation}

		Plugging  (\ref{eq:assm})  and (\ref{eq:commse}) into  (\ref{eq:comm3}), 
		taking into account  (\ref{eq:pinduce1}),  we obtain  immediately  the first  identity  of  (\ref{eq:pind}), which also implies  the  second one. This completes  the proof  of  Proposition  \ref{prop:funct2}.
	\end{proof}
	
	\begin{proposition}\label{prop:ass} Let $ T_i: \Xx_i \leadsto \Xx_{i+1}$  be   probabilistic morphisms  for $i =1, 2,3$. Then we have
		\begin{equation}\label{eq:ass}
		\overline{T_3\circ (T_2 \circ T_1)} = \overline{(T_3 \circ T_2) \circ T_1}.
		\end{equation}
	\end{proposition}
	\begin{proof} Proposition \ref{prop:ass} follows from known  properties  of Markov  kernels. Here  we give another  short algebraic (categorical)  proof. Straightforward computations  using (\ref{eq:compose}) and (\ref{eq:pinduce1})  yield
		\begin{equation}\label{eq:compose1}
		\overline{T_3 \circ (T_2 \circ T_1)} = P_*(T_3) \circ \overline{(T_2 \circ T_1)} = P_*(T_3) \circ P_*(T_2) \circ \overline{T_1},
		\end{equation}
		\begin{equation}
		\label{eq:compose2}
		\overline{(T_3 \circ T_2)\circ T_1} = P_*(T_3 \circ T_2) \circ \overline{T_1}.
		\end{equation}
		By Proposition \ref{prop:funct2}  the LHS of (\ref{eq:compose1})   equals the LHS of  (\ref{eq:compose2}). This  completes  the proof  of  Proposition \ref{prop:ass}.
	\end{proof}                    
	Combining  Proposition \ref{prop:funct2} with Propositions \ref{prop:ass}  and  \ref{prop:markov}, we prove  the    assertion  of  Theorem \ref{thm:cat}(1)   concerning the functoriality  of $(\delta, P_*)$ and $(\delta, M_*)$.
	
	Clearly  $\delta: \Xx \to \Pp (\Xx)$  is an injective map.  Using  the following  equality
	$$ S_*(T)(\delta_x) = \overline{T}   (x) \in \Pp (\Yy)$$
	for any $x\in \Xx$
	we conclude   that  $P_*$ and $M_*$  are injective   functors. This completes  the proof  of  Theorem \ref{thm:cat}(1).
	
	\
	
	2.  Theorem  \ref{thm:cat}(2)  has been  proved by Morse-Sacksteder \cite[Proposition 5.1]{MS1966}.   We now present an alternative  proof that is   a bit  shorter.  By  the     (\ref{eq:pinduce1})  and   the   validity  of  Theorem \ref{thm:cat}(2)  in the case $T$  is a   measurable  mapping, it suffices  to verify    that $ev_P(\nu) \ll ev_P (\mu)$  which is obvious.
	This completes  the proof  of Theorem \ref{thm:cat}.
\end{proof}


\begin{remark}\label{rem:diag}
	
(1)	Theorem  \ref{thm:cat} suggests  that   to study   the category $PMeas$  of   measurable  spaces  whose morphisms are probabilistic morphims  we can  consider  its embedding into  the  category  $(\Pp(\Xx),\Sigma_w)$  whose morphisms are
	generated  by measurable  mappings  between  $\Xx$  and $\Yy$, which  are    the restriction of   linear bounded  mappings  between  Banach spaces $\Ss(\Xx)$  and $\Ss(\Yy)$  endowed with the total variation norm.

	(2) Lawvere   introduced  the $\sigma$-algebra $\Sigma_w$  on $\Pp(\Xx)$ and defined the probabilistic  map $ev$ in terms of Markov kernels  \cite{Lawvere1962}.  He coined the term  of a ``probabilistic mapping" and noted  that it is equivalent to the notion of  a  Markov kernel.  He also  defined   $P_*(T)$     by  a formula equivalent to (\ref{eq:markov1}).  
	In \cite{Giry1982} Giry   considered the  category   of measurable spaces  whose morphisms  are measurable mappings  and the category  of Polish spaces $\Xx$  whose morphisms are continuous mappings. He noticed that  the space
	$(\Pp(\Xx), \tau_w)$ is again Polish  cf. \cite[Theorem 2.3]{GH1989}.
	He proved Theorem \ref{thm:cat}  for  the case that  $P$ is a measurable mapping and $\Xx$ is a Polish  space. 	The triple $(P_*, \delta, ev_P)$ is now called  {\it  Giry's probability monad}. 
	Giry noticed that the Kleisli category  of  the  probability 
	monad $(P_*,\delta, ev_P)$ is the category  of  measurable  spaces
	whose morphisms are  probabilistic  morphisms. We   refer the  reader  to \cite{FP2018}, \cite{FP2018} for further investigation  of probability monads. Panagaden \cite{Panangaden1999}, \cite{Panangaden2009}    extended  the Giry theory to the category SRel
	(stochastic  relations)  whose  objects  are measurable spaces and
	whose  morphisms  are transition  subprobabilities, but he did not use the concept of  $\sigma$-algebra on the space of   subprobabilities.
	Recently Fritz  extended   the  category  of  probabilistic morphisms    to the Markov category \cite{Fritz2020}, using high-level axioms,   following works of Golubtsov, especially \cite{Golubtsov2002}, and the work by Cho and Jacobs \cite{CJ2019}. Parzygnat  extended  Fritz's theory  further \cite{Parzygnat2020}, in particular,  to  include  quantum probabilities.

	(3)  Chentsov  called the   category  of  Markov kernels        {\it the   statistical category} and their morphisms -  {\it Markov morphisms} \cite{Chentsov1965, Chentsov1972}.
	Chentsov also  showed  that  $S_*(T)$ is a linear  operator  with  $|| S_*(T)|| =1 $  for any  Markov kernel $T$ and   that  $S_*(T)$  sends  a probability measure to a probability measure \cite[Lemma 5.9, p.72]{Chentsov1972}.   Markov morphisms $S_*(T)$   have been  investigated  further  in \cite{AJLS2015}, \cite{AJLS2017}, \cite{AJLS2018}.
	
	(4) Historically,   Blackwell   was  the first      who  considered    probabilistic morphisms, which he        called {\it stochastic mappings}, between parameterized statistical  models,   which  he called  statistical    experiments \cite{Blackwell1953}. LeCam  \cite{LeCam1964}  used    the  equivalent  terminology   {\it randomized mapping}    and  Chentsov \cite{Chentsov1965}  used   the equivalent  notion of   {\it transition  measure}.
\end{remark}

\subsection{Statistical models and sufficient  probabilistic morphisms}\label{sec:suff}

Given a probabilistic morphism $T$  (for instance,  a  measurable mapping $\kappa$), we shall also use the short notation $T_*$ and $\kappa_*$  for $M_*(T)$ and  $M_* (\kappa)$, respectively.

\begin{definition}
	\label{def:cats} {\it  A  statistical model}  is a pair $(\Xx, P_\Xx)$ where   $\Xx$ is a  measurable space and $P_\Xx\subset \Pp(\Xx)$.  {\it The category of statistical  models}    consists  of  statistical   models as its objects   whose   morphisms $\varphi: (\Xx, P_\Xx) \leadsto (\Yy, P_\Yy)$  are probabilistic  morphisms  $T:\Xx\leadsto \Yy$    such that $T_*(P_\Xx)\subset  P_\Yy$.  A morphism  $T: (\Xx, P_\Xx) \leadsto  (\Xx, P_\Xx)$ will be called {\it a unit},  if $T_*: P_\Xx \to P_\Xx$ is the  identity.  Two   statistical  models
	$(\Xx, P_\Xx)$ and $(\Yy, P_\Yy)$ are called  {\it equivalent}, if there   exist  morphisms  $T_1: (\Xx, P_\Xx) \leadsto (\Yy, P_\Yy)$   and $T_2: (\Yy, P_\Yy) \leadsto (\Xx, P_\Xx)$ such that  $T_1 \circ  T_2 $ and $T_2\circ T_1$ are units.  In this  case
	$T_1$  and $T_2$  will be called {\it equivalences}.
\end{definition}

\begin{remark}
	\label{rem:cats}  (1)  The Kleisli category of Giry's probability monad, also called  the Kleisli category  of probabilistic   morphisms,    can be realized  as  a
	subcategory  of the category  of   statistical models  by assigning  each   measurable space
	$\Xx$ the pair $(\Xx, \Pp(\Xx))$.  
	
	(2) The  notion of  statistical models  and their morphisms  in Definition \ref{def:cats} is almost equivalent  to the notion of   statistical    systems and their   morphisms, called  statistical operations,    introduced by Morse-Sacksteder \cite{MS1966}, except    that  Morse-Sacksteder  allowed     morphisms  that need  not to be  induced  from  probabilistic morphisms (they  considered  conditional probability as defining statistical operations),  and almost    coincides  with    the Markov category  of  family of probability  distributions introduced  by Chentsov \cite[\S 6, p. 76]{Chentsov1972},   except   that    Chentsov considered   only morphisms between    statistical models   whose   probability  sets $P_\Xx, P_\Yy$   are  parameterized   by   the same   set $\Theta$.
\end{remark}
Among       probabilistic morphisms
defining equivalences  between statistical models,  there  is an important  class  of   sufficient  morphisms,  introduced by Morse-Sacksteder, which  we reformulate as   follows.

\begin{definition}\label{def:suff} (cf. \cite{MS1966}) 
	A morphism  $T: (\Xx, P_\Xx) \leadsto (\Yy, P_\Yy)$  between statistical models   will be called {\it sufficient}  if  there exists a  probabilistic
	morphism $\underline \pb : \Yy \leadsto \Xx$ such that  for all $\mu \in P_\Xx$   and  $h \in L(\Xx)$ we have
	\begin{equation}
	\label{eq:suff}
	T_*(h \mu) = \underline{\pb}^*(h)T_*(\mu) \text{, i.e., }  \underline{\pb} ^* (h) = \frac{d T_* (h \mu)}{d T_* (\mu)}\in  L^1(\Yy, T_*(\mu)).
	\end{equation}
	In this case  we shall  call        $T: \Xx \leadsto \Yy$  {\it a  probabilistic morphism   sufficient  for $\Pp_\Xx$}  and   we  shall call  the  measurable  mapping $\pb: \Yy \to \Pp(\Xx)$  defining    the probabilistic  morphism $\underline {\pb}: \Yy \leadsto \Xx$
	{\it a conditional  mapping   for  $T$}.
\end{definition}

\begin{remark}\label{rem:suff0}  (1)  We call $\pb: \Yy \to \Pp (\Xx)$    a conditional mapping  because  of  the  interpretation  of $\pb$  as a regular  conditional   probability  in the case  $T$  is   a   measurable mapping, see  (\ref{eq:suffs})  below.
	
	(2) As in  Remark \ref{rem:sigmaw} (1),   we  note that    to  prove the sufficiency of  a probabilistic  morphism $T$   w.r.t. $P_\Xx \subset \Pp(\Xx)$  it suffices  to verify  (\ref{eq:suff})  for all $\mu \in P_\Xx$ and for  all  $h=  1_A$ where  $A \in \Sigma_\Xx$.
\end{remark}

\begin{example}
	\label{ex:suff}
	Assume  that $\kappa:(\Xx, P_\Xx) \leadsto (\Yy, P_\Yy)$  is a   sufficient morphism, where $\kappa: \Xx \to \Yy$  is a  statistic.  Let   $\pb:  \Yy \to \Pp (\Xx), \,  y \mapsto  \pb _y,$ be a conditional  mapping for $\kappa$. By (\ref{eq:adjoint1}),  $\underline\pb^* (1_A)(y)  =  \pb_y(A)$,  and  we  rewrite   (\ref{eq:suff})
	as follows
	\begin{equation}
	\label{eq:suffs}
	\pb_y(A)=  \frac{d\kappa_*(1_A\mu)}{d\kappa_*\mu}\in L^1(\Yy, \kappa_*(\mu)) .
	\end{equation}
	The RHS of (\ref{eq:suffs}) is  the conditional  measure of $\mu$  applied to  $A$  w.r.t. the  measurable mapping $\kappa$.  The   equality (\ref{eq:suffs})  implies  that this   conditional  measure is regular  and independent of  $\mu$. 
	Thus  the notion of sufficiency of $\kappa$  for  $P_\Xx$  coincides with the classical notion of   sufficiency  of      $\kappa$  for  $P_\Xx$, see e.g.,  \cite[p. 28]{Chentsov1972}, \cite[Definition 2.8, p. 85]{Schervish1997}.  We   also note that the equality  in (\ref{eq:suffs})  is understood  as  equivalence class  in $L^1(\Yy,  \kappa_*(\mu))$ and hence    every
	statistic $\kappa '$  that  coincides  with a    sufficient statistic   $\kappa$   except  on a zero $\mu$-measure   set,   for all  $\mu \in P_\Xx$,  is also   a sufficient  statistic  for  $P_\Xx$.
\end{example}

\begin{example}
	\label{ex:FN} (cf. \cite[Lemma 2.8, p. 28]{Chentsov1972}) Assume  that  $\mu \in \Pp(\Xx)$ has a regular  conditional    distribution    w.r.t.   to a statistic 
	$\kappa: \Xx \to \Yy$, i.e., there  exists  a measurable  mapping
	$\pb : \Yy \to \Pp(\Xx), y \mapsto \pb_y,$  such that
	\begin{equation}\label{eq:cond1}
	\E_\mu ^{\sigma(\kappa)}(1_A | y) = \pb_y(A) 
	\end{equation}
	for any $A \in \Sigma_\Xx$  and  $y \in \Yy$.  
	Let  $ \Theta$  be  a parameter set and  $P : =\{\nu_\theta  \in \Pp(\Xx) | \, \theta \in \Theta\}$  be   a  family of  probability measures  dominated  by $\mu$. If  there exist  a function
	$h : \Yy \times \Theta \to \R$  such  that  for all $\theta \in \Theta$  and  we have
	\begin{equation}\label{eq:NF}
	\nu_\theta = h(\kappa (x))\mu
	\end{equation}
	then $\kappa $ is sufficient  for  $P$, since    for any $\theta \in \Theta$ 
	$$\pb^* (1_A) = \frac{d\kappa_*  (1 _A  \nu_\theta)}{d\kappa_* \nu_\theta} $$
	does not depend on $\theta$.  The  condition (\ref{eq:NF})  is the Fisher-Neymann sufficiency condition  for a family  of  dominated
	measures \cite{Neyman1935}.   
\end{example}

Let us assume that  $T: (\Xx, P_\Xx) \leadsto  (\Yy, P_\Yy)$ is  a sufficient  
morphism
for $\mu \in \Pp(\Xx)$.  Then  for any $A \in \Sigma _\Xx$ we have
\begin{equation}
\label{eq:suff1}
\mu(A)  \stackrel{(\ref{eq:markov1})}{=}  T_*(1_A \mu) (\Yy)\stackrel{Theorem \, \ref{thm:cat}(2)}{=}\int_\Yy\frac{dT_*(1_A \mu)}{dT_*\mu}dT_*\mu \stackrel{(\ref{eq:suff})}{=} \int _\Yy  \pb_y  (A) d T_*(\mu).
\end{equation}
Comparing (\ref{eq:suff1})  with (\ref{eq:markov1}), we conclude   that  $\mu = \underline{\pb}_* \circ T_* (\mu)$.  Hence  $T_*\circ  \underline{\pb}_* (T_* \mu) = T_*\mu$.
It  follows  that $(\Xx, \mu)$ is  equivalent   to $(\Yy,  T_*(\mu))$. 
Hence  we  get the following
\begin{proposition}
	\label{prop:suffs}(\cite  [Proposition 5.2]{MS1966}) Two  statistical  models  $(\Xx, P_\Xx)$   and  $(\Yy, P_\Yy)$  are  equivalent  if there exists  a         probabilistic  morphism   $T: \Xx \leadsto \Yy$ such that  $T$ is sufficient for  $P_\Xx$    and  $T_*(P_\Xx) = P_\Yy$.   
\end{proposition}

\begin{remark}\label{rem:reg1} (1) The converse of Proposition \ref{prop:suffs} is  valid, if   $P_\Xx$   is a dominated   family \cite[Theorem 2.1]{Sacksteder1967},
	and   it    is   false  without  the  dominated family condition   \cite[\S 6]{Sacksteder1967}.

	(2) There  are many   characterizations of sufficiency  of a statistic,  see e.g. \cite[\S 2.9, p. 35]{Pfanzagl2017}. In   \cite[Proposition 5.6, p. 266]{AJLS2017}  Ay-Jost-L\^e-Schwachh\"ofer  give   a    characterization  of a sufficient  statistic    w.r.t.   to a {\it   regular}    set $P_\Xx$
	of  dominated measures  that is  parameterized  by a smooth   manifold  via       the notion of  ``information loss". 
\end{remark}

We  complete  this section with   the following theorem on   the category of
sufficient  morphisms.

\begin{theorem}\label{thm:sprob}   (1) A  composition of        sufficient    morphisms between     statistical models   is a   sufficient   morphism. 
	
	(2)  Assume that $T: (\Xx, P_\Xx) \leadsto (\Yy, P_\Yy)$ is  a sufficient  morphism   and 
	$T_*(\Pp_\Xx) = P_\Yy$. Let $\pb: \Yy \to  \Pp (\Xx)$ be a conditional mapping   for  $T$. Then  $\underline{\pb}: \Yy \leadsto \Xx$  is a  sufficient   probabilistic  morphism   for  $P_\Yy$ and $\overline{T}: \Xx \to \Pp(\Yy)$ is a conditional  mapping   for  $\underline{\pb}$.
\end{theorem}

\begin{proof} 1.   Assume  that $T_1 : (\Xx, P_\Xx)\leadsto (\Yy, P_\Yy)$    and $T_2: (\Yy, P_\Yy) \leadsto (\Zz, P_\Zz)$ are sufficient  morphisms   with     conditional mappings $\pb_1 : \Yy \to \Pp( \Xx)$    and  $\pb_2 : \Zz \to \Pp (\Yy)$ for $T_1$ and $T_2$,  respectively. Then we  have  for  any $h  \in L(\Xx)$ and  $\mu\in P_\Xx$
	$$ (T_2 \circ T_1 )_* (h\mu)=( T_2) _* \circ  (T_1) _* (h \mu)$$
	$$=  (T_2)_*(\underline{ \pb}_1 ^* (h) (T_1)_* (\mu)) =  \underline{\pb_2 }^*\circ \underline {\pb}_1^* (h) ( T_2 \circ T_1)_*(\mu)$$
	which   proves  the   assertion   (1)  of   Theorem  \ref{thm:sprob}.

	2. Assume   that  $T: (\Xx, P_\Xx) \to  (\Yy, P_\Yy)$ is a sufficient  morphism and
	$\Pp_\Yy =  T_*(P_\Xx)$. Let  $\pb: \Yy \to \Pp (\Xx)$  be a  conditional mapping for $T$, i.e., we  have  for any $h \in L(\Xx)$   and $\mu \in  P_\Xx$
	\begin{equation}\label{eq:suff2}
	\underline{\pb}^* (h) = \frac{dT_* (h \mu)}{dT_* (\mu)}.
	\end{equation}
	To   prove  the  assertion  (2)   of Theorem \ref{thm:sprob}  it suffices  to establish  the following   equality
	\begin{equation}\label{eq:repro}
	\underline{\pb}_* (1_B  T_* \mu) (A)=  T^* (1_B) \mu(A)
	\end{equation}
	for  any  $\mu \in P_\Xx$,  $B \in \Sigma_\Yy$, $A\in \Sigma_\Xx$ and  $h \in L(\Xx)$.
	
	Straightforward computations  yield
	\begin{eqnarray} 
	\underline{\pb}_* (1_B  T_* \mu) (A)\stackrel{(\ref{eq:markov1})}{=}  \int_\Yy \pb_y (A) d (1_B  T_*\mu)\nonumber\\
	= \int_B  \pb _y (A)  dT_*\mu \stackrel{(\ref{eq:suff2})}{=} \int_B  \frac{ dT_* (1_A\mu)}{dT_*\mu} dT_*\mu\nonumber \\
	=  T_* (1_A \mu) (B)\stackrel{(\ref{eq:markov1})}{= }\int_\Xx  \overline  T_x (B)  d  1_A\mu\nonumber  \\  
	=   \int  _A  \overline{T}_x  (B) d\mu. \label{eq:sprob1}
	\end{eqnarray}
	On the other  hand we have
	\begin{equation}\label{eq:sprob2}
	T^*(1_B)\mu(A) = \int_A T^* (1_B)    (x)  d\mu  \stackrel{ (\ref{eq:adjoint1})}{=} \int_A\overline T_x (B)d\mu.   
	\end{equation}
	Comparing  (\ref{eq:sprob1})  with  (\ref{eq:sprob2}), we obtain  (\ref{eq:repro}).
	This  completes  the proof of Theorem \ref{thm:sprob}.
\end{proof}

\begin{remark}\label{rem:suffs} (1) Theorem \ref{thm:sprob}  implies   that in the  subcategory  of     statistical models  equivalent to a given statistical  model, whose morphisms  are 
	sufficient morphisms,  every       object  is a  terminal object, so we don't have   the notion  of a ``minimal sufficient     probabilistic  morphism"     like the notion of   a minimal sufficient    statistic.   Note that  a minimal sufficient  statistic  does not always  exist  \cite[\S 2.6, p. 24]{Pfanzagl2017}.
	
	(2) We would  like to point out that  Fritz \cite{Fritz2020}     treats  sufficient  statistics  in  abstract  Markov category and  reproves  several classical  results  concerning sufficient  statistics   using  a categorical  approach.
	
	(3)  We refer  the reader  to \cite{Le2020} for further applications of probabilistic  morphisms  to
the	category of statistical models  and their  natural geometry.
	
\end{remark}

\section{Probabilistic morphisms and Bayes' formula}\label{sec:bayes}
In this section we formalize  the notion  of  a parameter space of a  probabilistic model in Bayesian statistics  via the notion  of a Bayesian statistical model (Definition \ref{def:bstatmodel}, Example \ref{ex:bstat}, Remark \ref{rem:bayes}).   We   define  the  notion  of  posterior   distributions
(Definition \ref{def:post})   and   express   the existence   of   posterior   distributions  in terms  of   equivalence  of statistical  models   (Proposition \ref{prop:post}). 
Then we prove our Main Theorem~\ref{thm:bayesm} which provides a new formula  for the posterior  distributions. We  derive  a  consequence (Corollary \ref{cor:lopital}) and  prove the consistence  of our  formula  (Proposition \ref{prop:cons}).

\subsection{Bayesian statistical models  and probabilistic  morphisms}\label{subs:bsm}
\begin{definition}
	\label{def:bstatmodel}  A  Bayesian statistical model   is  a quadruple  $(\Theta,  \mu_\Theta, \pb , \Xx)$, where 
	$\pb: \Theta \to \Pp(\Xx)$ is a measurable mapping and $\mu_\Theta \in \Pp(\Theta)$,  called  a {\it   prior   measure}. 
\end{definition}

Rephrasing, a  Bayesian statistical model  is  defined  by
a  probabilistic  morphism $\underline{\pb}$  between  statistical models $(\Theta, \{ \mu_\Theta \})$ and $(\Xx,\Pp(\Xx))$.

Recall that, in our paper,    if  $\Xx$ is  a  topological space  then   $\Sigma_\Xx$  is assumed to be 
the Borel  $\sigma$-algebra $\Bb(\Xx)$  and  hence $\Pp(\Xx)$ consists of    Borel measures,  if not otherwise specified.

\begin{example}\label{ex:bstat}
	
	(1)  Let   $\Ss(\Xx)_{TV}$ denote  the Banach space  $\Ss(\Xx)$ endowed with  the       total variation norm  and $d_H$ denote the Hellinger distance  on $\Pp(\Xx)$, so the  topology   on $\Pp(\Xx)$  induced  by $d_H$   coincides with the topology  on $\Pp(\Xx)$  induced      from  the natural inclusion  of $i:\Pp(\Xx) \to \Ss(\Xx)_{TV}$. Assume  that    $(M, \mu_M)$ is a  smooth finite dimensional
	manifold   endowed  with a volume element $\mu_M$  and $\pb: M \to (\Pp (\Xx), d_H)$ is a  continuous map, for instance,  if $\pb$ is a smooth map, i.e.,
	the composition $i\circ \pb: M \to \Ss(\Xx)_{TV}$ is a  smooth map, see  e.g., \cite[p. 168]{AJLS2017}, \cite[p. 380]{Bogachev2010}. We shall show that $(M, \mu_M, \pb, \Xx)$ is a Bayesian   statistical model, i.e.,   we shall show  that  the  map
	$\pb: (M, \Bb(M)) \to  (\Pp(\Xx), \Sigma_w)$  is measurable. Denote  by $\tau_\rho$ 	the  smallest topology  on $\Ss(\Xx)$ 
	such that     for all  $f \in \Ff_s (\Xx)$  the  map $I_f: \Ss(\Xx) \to \R$    is continuous.  Clearly  the  topology  $\tau_\rho$ is  weaker   than  the  strong topology
	generated by  the    total variation norm.  Let $\tau_\rho$  also denote  the  induced
	topology  on $\Pp(\Xx)$. It follows  that  the   map $\pb: M \to \Pp(\Xx)$  is  $\tau_\rho$-continuous. Hence the  map $\pb : (M, \Bb(M)) \to  (\Pp (\Xx),\Bb (\tau_\rho))$ is measurable. It is not hard to    check that $\Sigma_w \subset  \Bb (\tau_\rho)$,   and  moreover   $\Bb(\tau_w) \subset \Bb(\tau_\rho)$  with the  equality  $\Bb(\tau_w) = \Bb(\tau_\rho)$  if and only if $\Xx$  is countable \cite{GH1989}. We conclude  that $\pb: (M, \Bb(M)) \to  (\Pp (\Xx), \Sigma_w)$ is a measurable  mapping.

	(2) For any $\mu_\Pp \in \Pp^2(\Xx)$  the quadruple $(\Pp(\Xx), \mu_\Pp, Id_\Pp, \Xx)$ is a Bayesian 
	statistical model. 
	
	(3)   Let $(\Theta, \mu_\Theta, \pb, \Om_\infty)$ be a  Bayesian  statistical model, where  $\Om_\infty : = \{ \om_1, \cdots, \om_\infty\}$.  Then   the probabilistic morphism  $\underline{\pb} : \Theta \leadsto \Om_\infty$   will be called  a {\it   random  partition}   of  $\Theta$.  If $\underline{\pb}$ is a measurable mapping then
	it is  also called {\it  a measurable   partition}.
	A random   partition $\underline\pb: \Theta \leadsto \Om_\infty$ will be called {\it hierarchical} if  
	$\underline{\pb} = \underline{\pb}_k \circ \cdots \circ \underline{\pb}_1$,  where  $\underline{\pb}_1: \Theta \leadsto  \Om_\infty$,  and   $\underline{\pb}_i : \Om_\infty \leadsto \Om_\infty$ for   $i \in [2, k]$ are  probabilistic  morphisms. 
\end{example}

\begin{remark}
	\label{rem:bayes}  (1)   We recover the classical (frequentist)  definition of a  (parameterized)  statistical model  from Definition \ref{def:bstatmodel}  by  applying the ``forgetful  functor"  from the category of measurable  spaces  to the category of sets.
	
	(2) Definition \ref{def:bstatmodel} encompasses both  parametric  Bayesian  models, which  usually   consist of conditional  density functions $f(x|\theta)$, see
	e.g. \cite[p. 4]{Berger1993},  and nonparametric  Bayesian models. It is essentially equivalent to  the concept  of  a Bayesian  parameter space  defined in \cite[p. 16]{GR2003} but we use the more  compact  functorial language  of probabilistic morphisms.
\end{remark}

\subsection{Posterior  distributions and probabilistic morphisms}\label{subs:postprop}

Let   $(\Theta, \mu_\Theta, \pb, \Xx)$ be  a Bayesian   statistical model.  Then 
we  define  the joint distribution  $\mu$  on the  measurable space  $(\Theta \times \Xx , \Sigma_\Theta \otimes  \Sigma_\Xx)$ as follows 
\begin{equation}\label{eq:dis2}
\mu(B \times A) := \int_{B} \pb_\theta(A)  d\mu_ \Theta  \text{ for  all } B \in \Sigma_\Theta, \: A \in \Sigma_\Xx,
\end{equation}
where   we re-denote  $\pb(\theta)$ by $\pb_\theta$  for interpreting $\theta$ as  parameter  of  the distribution.  In  other words, denoting by
$\Gamma_{\pb}: =  (Id_{\Theta}, \pb):\Theta \leadsto (\Theta \times  \Xx)$  the graph of the probabilistic morphism $\pb$, then 	
$\mu =(\Gamma_{\pb})_*\mu_\Theta.$
As in  (\ref{eq:xiy}),  $\Pi_\Theta$  and   $\Pi_\Xx$  denote  the  measurable  projections  $(\Theta \times \Xx, \Sigma_\Theta \otimes \Sigma_\Xx)\to (\Theta, \Sigma_\Theta)$ and  $(\Theta \times \Xx, \Sigma_\Theta \otimes \Sigma_\Xx)\to (\Xx, \Sigma_\Xx)$  respectively. Using  $\Pi_\Xx$ we express the marginal  probability  measure $\mu_\Xx$  on $\Xx$   as follows:  $\mu_\Xx = (\Pi_\Xx)_* (\mu)$.   In other  words  we have $\mu_\Xx = \pb_* (\mu_\Theta)$.  
Clearly $\mu_\Theta  = (\Pi_\Theta)_* (\mu)$. We note that
$\mu_{\Xx|\Theta}(A|\theta): = \pb_\theta(A)$ is  the conditional
probability    of the  measure $\mu$ on  $(\Theta \times \Xx, \Sigma_\Theta \otimes \Sigma_\Xx)$  w.r.t.  the projection $\Pi_\Theta$,   since we have  the following equality
\begin{equation}\label{eq:bayes1}
\pb_\theta (A)= \frac{d(\Pi_\Theta)_* (1_{\Theta \times A}\mu)}{d(\Pi_\Theta)_*\mu}(\theta),
\end{equation}
cf. (\ref{eq:xiy}).

Formula (\ref{eq:bayes1})  motivates  the following  definition.

\begin{definition}
	\label{def:post}  A family of  probabilistic   measures  $\mu _{\Theta| \Xx}(\cdot|x)  \in\Pp (\Theta)$,  $x \in \Xx$,     is   called  a family  of  {\it   posterior   distributions  of $\mu_\Theta$ after seeing the data $x$}  if  for all $B \in \Sigma_\Theta$ we have
	\begin{equation}\label{eq:bayes2}
	\mu_{\Theta|\Xx} (B| x)   = \frac{d(\Pi_\Xx)_* (1_{B\times\Xx}\mu)}{d(\Pi_\Xx)_*\mu}(x),
	\end{equation}
	where  (\ref{eq:bayes2}) should  be understood  as  an    equivalence 
	class  of functions  in $L^1(\Xx,\mu_\Xx )$. 
\end{definition}

\begin{remark}\label{rem:post} (1)  Definition    \ref{def:post}  coincides   with    the  definition  of a  posterior   distribution in  classical   Bayesian statistics see e.g. \cite[\S 4.2.1, p. 126]{Berger1993},
	\cite[p. 16]{Schervish1997}.   Note that  in
	the both definitions  of  the mentioned  books     the authors   did not   explicitly    require  that
	$\mu_{\Theta|\Xx} (\cdot |x)$  must be a  $\sigma$-additive  function  on $\Sigma_\Theta$.  The last requirement  is   trivially satisfied when  one considers  only  Bayesian statistical models
	of dominated measures, i.e.  there exists  a measure $\nu_0 \in \Pp(\Xx)$ such  that  for all $\theta \in \Theta$  we have  $\pb(\theta)\ll \nu_0$.
	In  other papers  and books, e.g.,   in \cite{Ferguson1973}  and \cite{GV2017},
	statisticians  also   think of   posterior  distributions as regular    conditional distributions, which  is  equivalent  to   Definition  \ref{def:post}.  
	The   existence  of     posterior  distributions   is therefore equivalent to the existence  of regular  conditional     distributions.

	(2) Alternatively,  we can also consider  the  existence  of  posterior  distributions  as  the  existence  of disintegration  \cite{CP1997}, \cite[p. 380, vol.2]{Bogachev2007}, \cite{CJ2019}, \cite{Parzygnat2020}.

\end{remark}	
The following    Proposition generalizes \cite[Theorem 2.1]{Parzygnat2020}, which considers  discrete   spaces.

\begin{proposition}\label{prop:post}  Given a  Bayesian  statistical model  $(\Theta, \mu_\Theta, \pb, \Xx)$.   A    family of posterior  distributions
	$\mu_{\Theta|\Xx} (\cdot |x)$ exists  if and  only  if  the
	statistical  models  $(\Theta, \{\mu_\Theta\})$  and  $(\Xx, \{\mu_\Xx\})$
	are equivalent, where  $\mu_\Xx$ is the  marginal     distribution
	of $\Xx$, i.e.,  $\mu_\Xx = \pb_* (\mu_\Theta) \in \Pp (\Xx)$.
\end{proposition}

\begin{proof}   First  we assume that  a family of posterior   distributions
	$\mu_{\Theta|\Xx} (\cdot |x)$ exists. As before   we let $\mu: = (Id_\Theta, \pb)_*\mu _\Theta  \in \Pp(\Theta \times \Xx)$. Then we    define   a probabilistic  morphism  $Q: \Xx \leadsto \Theta$  whose   generating  measurable  mapping $\overline Q: \Xx \to \Pp (\Theta)$ is  defined  by
	\begin{equation}\label{eq:q}
	\overline Q(x) : = \mu_{\Theta|\Xx} (\cdot |x).
	\end{equation}
	Let us compute  for $A \in \Sigma_\Xx$ and  $B \in \Sigma_\Theta$
	\begin{eqnarray}\label{eq:inv}
	\bigl((\Gamma_Q)_*\mu_\Xx\bigr) (B \times A)=   \int_A  Q(x) (B) d\mu_\Xx= \int_A \frac{ d (\Pi_\Xx)_*  (1_{\Bb \times \Xx}\mu) }{d(\Pi_\Xx)_* \mu} d\mu_\Xx\nonumber\\
	= \int_{\Theta\times  A}  \frac{ d (1_{B \times  \Xx}\mu)}{ d\mu} d\mu = \mu(B \times  A).
	\end{eqnarray}
	It  follows  that  $(Q)_* \mu_\Xx = (\Pi_\Theta)_* (\mu)= \mu_\Theta$.
	Hence    the statistical models  $(\Theta, \mu_\Theta)$ and $(\Xx, \mu_\Xx)$ are equivalent.
	
	Now assume  that   the statistical models  $(\Theta, \mu_\Theta)$ and $(\Xx, \mu_\Xx)$ are equivalent. Then  there  exists  a  probabilistic  morphism
	$Q: \Xx \leadsto \Theta$ such  that  $Q_*  (\mu_\Xx) = \mu_\Theta$. It follows that
	$\pb \circ  Q_*  (\mu_\Xx) = \mu_\Xx$.
	Then we have
	\begin{equation}\label{eq:compo2}
	\mu= (Id_\Theta,  \pb)_* \mu_\Theta = (Id _{\Theta}, \pb_* )(Q_* \mu_\Xx)= (Q, \pb \circ Q)_* \mu_\Xx = (Q,    Id_{\Xx})_*\mu_\Xx.
	\end{equation}
	It follows  from (\ref{eq:compo2}) and  (\ref{eq:bayes1})  that  $ \mu_{\Theta |\Xx} (B|x) : =\overline  Q (x)(B), \, x\in \Xx,$ is the  required    family of   posterior   distributions. This  completes
	the proof  of  Proposition \ref{prop:post}. 
\end{proof}

\begin{remark}\label{rem:epost} One knows  several  theorems  on the existence of regular    conditional   distributions, see  e.g. \cite{CP1997}, \cite[Chapter 10, vol. 3]{Bogachev2007}  for surveys.   To our knowledge  an explicit  formula     for posterior  distributions
	is given    only  in  the  classical  Bayes  formula  and in a few  other  exceptional    cases.  The difficulty  in deriving  an explicit  formula
	for posterior  distributions  is caused  by three facts. Firstly, the   conditional   probability  is expressed  in terms  of the Radon-Nikodym derivative,  whose  existence
	is known  but an  explicit  formula    for    the Radon-Nikodym    derivative   is given  only in few   exceptional cases, see \cite{JLT2020}.  Secondly,    if we   have    an explicit formula  for   $\mu_{\Theta|\Xx} (B|x)$  satisfying  (\ref{eq:bayes2})
	we need  to show   that   for  $\mu_\Xx$-almost  all $x\in \Xx$,  $\mu_{\Theta|\Xx} (B|x)$ defines a $\sigma$-additive   function  on  $\Sigma_\Theta$.  Finally  to make    a   posterior  distribution formula  suitable for Baeysian   inference  we need  to     show that  the  singular   null set  that  appears  in  the  formula  remains     a
	null set  in the updated  marginal  measure.
\end{remark}

\subsection{Proof of   Theorem \ref{thm:bayesm}}\label{subs:bf}  Our proof of Theorem \ref{thm:bayesm}  shall be  divided in  two steps.  In the first  step
we shall   prove Theorem  \ref{thm:bayesm} for  $k =1$. In the  second step  we   shall          prove Theorem \ref{thm:bayesm} for any $k \in \N^+$.

\underline{Step 1}.  In this  step, assuming $k =1$,   we  shall    establish   the  existence  of
a  null set  $S\subset  \Xx$       that satisfies  the  condition
of Theorem  \ref{thm:bayesm}.  This will be done  by combining Proposition \ref{prop:der1},  Corollary \ref{cor:integ1}  below    with   the     existence of a countable   algebra generating  the Borel $\sigma$-algebra   of  a Souslin space $\Theta$  and a certain  property of  $\Theta$ (Proposition \ref{prop:souslin} below).  
In what follows we use shorthand  notations $\nu_1$  for the  measure $\mu_\Xx = (\Pi_\Xx)_*\mu$ and
$\nu_2$  for $ (\Pi_\Xx)_* (1_{B\times \Xx}\mu)$.  Clearly
$\nu_2 = \pb_*(1_B\mu_\Theta)\in \Mm(\Xx)$. By Theorem \ref{thm:cat}(2)
$\nu_2 \ll \nu_1$.
We recall  that  
\begin{equation}\label{eq:update3}
\nu_1(D_r(x)) = \int_\Theta\pb_\theta (D_r(x))d\mu_\Theta,
\end{equation}
\begin{equation}
\label{eq:update4}
\nu_2(D_r(x))= \int _B \pb _\theta (D_r(x))d\mu_\Theta.
\end{equation}

For  any $x \in \Xx$  we     set
$$ \overline{D}_{\nu_1} \nu_2  (x): = \lim _{r \to 0} \sup \frac{\nu_2 (D_r (x))}{\nu_1(D_r(x))} \text{  and  } \underline{D}_{\nu_1} \nu_2 (x) := \lim_{r\to 0} \inf \frac{\nu_2 (D_r (x))}{\nu_1(D_r(x))}$$
where  we set  $ \overline{D}_{\nu_1} \nu_2  (x)= \underline{D}_{\nu_1} \nu_2 (x) = +\infty$
if $\nu_1(D_r(x))=0$   for some $r>0$.

Furthermore  if $\overline{D}_{\nu_1} \nu_2 (x) = \underline{D}_{\nu_1}\nu_2 (x)$ then we set
$$D_{\nu_1}\nu_2 (x) : = \overline{D}_{\nu_1} \nu_2 (x) = \underline{D}_{\nu_1}\nu_2 (x)$$
which  is called {\it the derivative  of $\nu_2$ with respect  to $\nu_1$ at $x$.}
\begin{proposition}\label{prop:der1}  There  is a  measurable subset  $S_B\subset \Xx$ of zero $\nu_1$-measure   such that   for any $x \in \Xx \setminus S_B$ the 
	function $D_{\nu_1}\nu_2$  is well-defined.  Setting  $\tilde D_{\nu_1}\nu_2(x): =0$  for $x \in S_B$  and  $\tilde D_{\nu_1} {\nu_2} (x)= D_{\nu_1} \nu_2(x)$ for $x\in \Xx \setminus S_B$. Then the function  $\tilde D_{\nu_1}{\nu_2}:\Xx \to \R$ is measurable and serves as the Radon-Nikodym  density of the measure  $\nu_2$ with respect  to 
	$\nu_1$.
\end{proposition}
\begin{proof}
	We note that  the inclusion  map $\Xx \to (M^n, g)$ is continuous  and hence
	measurable. Hence  the  measures $i_*(\nu_1)$ and $i_*(\nu_2)$ are  Borel  measures  on $(M^n,g)$
	and $i_*(\nu_2) \ll i_*(\nu_1)$, since  $\nu_2 \ll \nu_1$.  Now   Proposition \ref{prop:der1} follows immediately  from Theorem 1.2 in \cite{JLT2020}.
\end{proof}

Proposition  \ref{prop:der1}  implies that  for  each $B \in \Sigma_\Xx$  
there exists  a   measurable  subset $S_B \subset \Xx$ of  zero $\mu_\Xx$-measure   such that the RHS  of (\ref{eq:updatem})  is well  defined  for   all $x \in \Xx \setminus S_B$,  and  moreover   it coincides  with the RHS  of (\ref{eq:bayes2}) as functions in $L^1(\Xx, \mu_\Xx)$.  We express  this fact in the following    equivalent   statement.

\begin{corollary}\label{cor:integ1} Let $B \in \Sigma_\Theta$.   For $x \in  (\Xx \setminus  S_B)$ let  $p (B, x):= \mu_{\Theta|\Xx} (B|x)$  that is defined  in (\ref{eq:updatem}) and  for $x\in  S_B$ we let  $ p (B,x): =0$.  Then  the function  $p(B, x) : \Xx \to \R$  is measurable. Furthermore,   for  any $A \in \Sigma_\Xx$  we have
	\begin{equation}\label{eq:integ1}
	\mu (B \times  A)= \int _{ A} p(B, x)d\mu_\Xx (x).
	\end{equation}
\end{corollary}

The strategy  of the remainder  of the  proof of the first assertion   of Theorem \ref{thm:bayesm}  for $k =1$  (Step 1) is  as  follows.  
By  \cite[Corollary  6.7.5, p.  25, vol.2]{Bogachev2007}, there exists  a countable  algebra $\Uu_\Theta$ 
generating  $\Bb(\Theta)$.  In Lemma \ref{lem:additive}  we   shall   show  the existence
of a  subset $S_0 \subset \Xx$  of  zero $\mu_\Xx$-measure  such   that
for  any ${ x} \in (\Xx \setminus S_0)$  and  any $B \in \Uu_\Theta$    the  RHS  of (\ref{eq:updatem})  is well-defined  and  $\mu_{\Theta|\Xx} (\cdot|{ x}) = p(B, {\bf x})$, moreover $\mu_{\Theta|\Xx}(\cdot| { x})$ is an additive   function on  $\Uu_\Theta$.     Then  in  Proposition \ref{prop:sigmad} 
we shall  apply  Proposition  \ref{prop:additive} below  to  show  that  there is 
a subset $S_1 \subset  (\Xx \setminus S_0)$ of zero  $\mu_\Xx$-measure such that  for  all  ${ x} \in \Xx \setminus (S_0 \cup S_1)$ the function
$p(B, { x})=\mu_{\Theta|\Xx}(\cdot| { x}) $ is  $\sigma$-additive  on  $\Uu_\Theta$, and hence   it  defines
a  probability  measure  on $\Theta$, since  $p(\Theta, { x}) = 1$ for  all ${\bf x} \in (\Xx \setminus (S_0 \cup S_1))$.  Since on $(\Xx \setminus  (S_0 \cup S_1))$ the
function $p(B, { x})$  coincides  with  $\mu_{\Theta|\Xx} (B,{ x})$  this shall complete  the proof   the first assertion  of Theorem \ref{thm:bayesm} for $k =1$.

We recall  that  a  family  $\Kk$ of subsets    of a set $X$  is called a {\it  compact  class}, if  for any  sequence $K_n$  of its elements  with $\cup_{ n=1}^ \infty  K_n = \emptyset$,  there  exists  $N$ such that $\cup_{i=1}^N K_n = \emptyset$ \cite[Definition 1.4.1, p. 13, vol.1]{Bogachev2007}.

\begin{proposition}
	\label{prop:additiven}\cite[Theorem 1.4.3, p. 13, vol. 1]{Bogachev2007} Let $\mu$ be  a nonnegative additive set function  on an algebra $\Aa$. Suppose  that there exists a compact class  $\Kk$ approximating $\mu$ in the following sense: for every $A \in \Aa$ and every $\eps > 0$, there  exist $K_\eps  \in \Kk$  and $A_\eps \in \Aa$  such  that $A_\eps \subset  K_\eps \subset  A$ and $\mu (A\setminus A_\eps)< \eps$. Then $\mu$ is countably additive.  In particular  this is true  if the compact class  $\Kk$ is contained  in $\Aa$  and for any $A \in \Aa$  one has the  equality
	$$ \mu(A) = \sup_{ K \subset  A, K \in \Kk} \mu(K).$$
\end{proposition}

{\it Continuation of Step 1}.  By Proposition \ref{prop:souslin}   any  family $\Kk$ of compact sets in $\Theta$ is a  {\it compact class}.
Propositions \ref{prop:souslin}  and \ref{prop:additive}  imply that  there  exists   a countable  algebra $\Uu_\Theta$  generating $\Bb(\Theta)$ such that $\Uu_\Theta$   contains a  countable union $\Kk_\Theta$ of metrizable  compact sets   on which the measure
$\mu_\Theta$ is  concentrated, and  for every $B \in \Bb(\Theta)$  and every  $\eps >0$ there
exists   a metrizable compact  set $K_\eps$  such that $\mu(B \setminus K_\eps)< \eps$. 
Thus 
the condition  of  \cite[Theorem 10.4.5 (ii), p. 359, vol.2]{Bogachev2007}  is  satisfied   for $\Aa =  \Uu_\Theta$.

Since $\Theta$ is a  Souslin space, by  \cite[Corollary  6.7.5, p.  25, vol.2]{Bogachev2007}, there exists  a countable  algebra $\Uu_\Theta$ 
generating  $\Bb(\Theta)$.  Let  $\Uu_\Theta$ consist   of countably many  sets $B_n\in \Sigma_\Theta$.

\begin{lemma}\label{lem:additive}  There exists  a measurable   subset  $S_0 \subset  \Xx $    of  zero  $\mu_\Xx$-measure   such  that   for all $x \in  (\Xx \setminus S_0)$  the   function
	$  \mu_{\Theta|\Xx}(\cdot|x)$  is additive  on $\Uu_\Theta$, moreover
	$\mu _{\Theta|\Xx}(\cdot| x) = p (B, x)$.
\end{lemma}
\begin{proof}   
	Let  $S_0 = \cup _iS_{B_i}$.   Then $S_0$  is measurable  and  it has   zero measure  set. By   Corollary \ref{cor:integ1},  $\mu_{\Theta|\Xx}(\cdot| x) = p (B, x)$  for
	any  $x\in (\Xx \setminus  S_0)$. 
	Clearly  the   RHS   of (\ref{eq:updatem})  is  additive for on $\Uu_\Theta$   for any $x \in (\Xx \setminus S_0)$. This completes  the proof  of Lemma \ref{lem:additive}
\end{proof}

Now  for all $ x\in \Xx$ we    define   a
set  function $\tilde \mu_{\Theta|\Xx} (\cdot |x)$ on  $\Uu_\Theta$  by setting 

\begin{equation}\label{eq:mtxm}
\tilde  \mu_{\Theta |\Xx} (\cdot |x) :=
\begin{cases}
0  \text{ for  } x\in   S_0,\\
p (\cdot |x)  \text{  for }    x\in \Xx \setminus S_0.
\end{cases}
\end{equation}
By Corollary \ref{cor:integ1}  and  Lemma \ref{lem:additive},  for  any $B \in \Uu_\Theta$ 
the  function $\tilde \mu_{\Theta|\Xx}(B|\cdot ): \Xx  \to \R$ is   measurable.

\begin{proposition}\label{prop:sigmad} There exists  a  measurable   subset $S_1 \subset  \Xx\setminus   S_0$  of zero $\mu_\Xx$-measure   such that  for any  $ x\in  (\Xx \setminus (S_0  \cup S_1))$  the set function $\tilde  \mu_{\Theta|\Xx}(\cdot| x)$ is $\sigma$-additive  on $\Uu_\Theta$.
\end{proposition}

For the proof   of Proposition  \ref{prop:sigmad} we   shall  use   Propositions    \ref{prop:additive},  \ref{prop:souslin}  below. 

We recall  that  a  family  $\Kk$ of subsets    of a set $X$  is called a {\it  compact  class}, if  for any  sequence $K_n$  of its elements  with $\cup_{ n=1}^ \infty  K_n = \emptyset$,  there  exists  $N$ such that $\cup_{i=1}^N K_n = \emptyset$. In particular, an arbitrary family of compact sets in 
a topological space  is  a compact class \cite[p. 13, vol.1]{Bogachev2007}.

\begin{proposition}
	\label{prop:additive}\cite[Theorem 1.4.3, p. 13, vol. 1]{Bogachev2007} Let $\mu$ be  a nonnegative additive set function  on an algebra $\Aa$. Suppose  that there exists a compact class  $\Kk$ approximating $\mu$ in the following sense: for every $A \in \Aa$ and every $\eps > 0$, there  exist $K_\eps  \in \Kk$  and $A_\eps \in \Aa$  such  that $A_\eps \subset  K_\eps \subset  A$ and $\mu (A\setminus A_\eps)< \eps$. Then $\mu$ is countably additive.  In particular  this is true  if the compact class  $\Kk$ is contained  in $\Aa$  and for any $A \in \Aa$  one has the  equality
	$$ \mu(A) = \sup_{ K \subset  A, K \in \Kk} \mu(K).$$
\end{proposition}

\begin{proposition}\label{prop:souslin}(\cite[Theorem 7.4.3, p. 85, vol. 2]{Bogachev2007})  Every  Borel  measure $\mu$  on a Souslin space  $\Theta$ is Radon and concentrated  on a countable  union  of metrizable compact sets. In addition, for every $B \in \Bb(\Theta)$  and every  $\eps >0$ there
	exists   a metrizable compact  set $K_\eps$  such that $\mu(B \setminus K_\eps)< \eps$.
\end{proposition}

\begin{proof}[Proof of Proposition \ref{prop:sigmad}]
	By Proposition \ref{prop:souslin}   there exists  a   compact class
	$\Cc$  consisting  of  compact  sets $C_{n,l}$   such  that  for all
	$n, l \in \N^+$ we have
	\begin{equation}\label{eq:bog1}
	C_{n,l}\subset  B_n  \text{ and  } \mu_\Theta (B_n \setminus C_{n,l})< 1/l.
	\end{equation}
	\begin{lemma}\label{lem:claim}  There  exists  a  measurable   subset   $S_1 \subset \Xx\setminus S_0$ of zero $\mu_\Xx$-measure such that  for  all $  n,l \in \N^+$   and  for all $x\in (\Xx \setminus (S_0 \cup S_1))$  and  we have
		\begin{equation}\label{eq:sup1}
		\tilde \mu_{\Theta|\Xx}(B_n | x) = \sup _l \tilde \mu_{\Theta|\Xx} (C_{n, l}|x) .
		\end{equation}
	\end{lemma}
	\begin{proof}  We  define  a function $q_n : \Xx \to \R$ by  setting $q_n (x)$ equal to the  RHS  of  (\ref{eq:sup1}),  if $x \in (\Xx \setminus S_0)$  and $ q_n (x) = 0$ otherwise.  Then $q_n : \Xx \to \R$ is measurable, since  $\tilde \mu_{\Theta|\Xx}(C_{n, l}| \cdot ): \Xx \to \R$  is measurable.
		Since $C_{n,l} \subset B_n$ for all $l$, we  have  for all $ x \in (\Xx \setminus S_0)$,    
		\begin{equation}\label{eq:bog2}
		q_n(x) \le \tilde \mu_{\Theta|\Xx} (B_n| x)  \text{ for all } x \in (\Xx \setminus S_0).
		\end{equation}
		Since $\mu_{\Theta|\Xx}  (C_{n,l}|x) \le q_n (x)$ for  $x \in \Xx \setminus S_0$, and  $S_0$ is a  measurable subset of  zero $\mu$-measure, taking into account (\ref{eq:integ1})  and (\ref{eq:mtxm}), we  have
		\begin{equation}\label{eq:bog3}
		((\Gamma_{\pb})_*\mu_\Theta)(C_{n,l} \times  \Xx) = \int_{\Xx} \tilde \mu_{\Theta|\Xx} (C_{n,l}|x) d\mu_\Xx \le  \int_{\Xx}  q_n (x)d(\mu_\Xx).
		\end{equation}
		Taking  into   account (\ref{eq:bog2})  and (\ref{eq:bog1}), we obtain from (\ref{eq:bog3})
		\begin{eqnarray}
		\sup _{l} ((\Gamma_{\pb})_*\mu_\Theta)(C_{n,l} \times \Xx)\le  \int_{\Xx} q_n (x)d(\mu_\Xx)\nonumber\\
		\le \int_{\Xx} \tilde\mu_{\Theta|\Xx} (B_n|x)d(\mu_\Xx)  = ((\Gamma_{\pb})_*\mu_\Theta)(B_n \times \Xx).\label{eq:bog4}
		\end{eqnarray}
		By (\ref{eq:bog1}), the far LHS  of (\ref{eq:bog4})  is equal  to the  far RHS of  (\ref{eq:bog4}). Taking into account (\ref{eq:bog2}), we conclude that   there    exists a 
		subset  $S_1\subset \Xx\setminus S_0$  of  zero $\mu_\Xx$-measure such that  $q_n (x)= \mu_{\Theta|\Xx} (x)$ for all $x \in \Xx \setminus (S_0\cup  S_1)$.  Since   both the functions  $q_n$ and
		$\mu_{\Theta|\Xx}$ are measurable, the  subset  $S_1$  is measurable. This  proves
		Lemma \ref{lem:claim}.
	\end{proof}
	
	Lemma \ref{lem:claim}  implies  that   for  all  $ x\in \Xx \setminus (S_0 \cup S_1)$  the additive
	function $\tilde \mu_{\Theta|\Xx} (\cdot |x)$  on the algebra $\Uu_\Theta$   has the  property
	that    the compact  class $\Kk_\Theta$  approximates  $p(\cdot, x)$ on
	$\Uu_\Theta$. Combining  with Proposition \ref{prop:additive}, this completes  the proof of Proposition \ref{prop:sigmad}.  
\end{proof}

Since $\tilde \mu_{\Theta|\Xx} (\cdot|x)$ is   $\sigma$-additive on $\Uu_\Theta$,  it  extends  to  a   measure  on $\Theta$.  Since $\tilde \mu_{\Theta|\Xx}(\Theta, x) = 1$  for any $x \in \Xx$  this completes  the proof of Theorem  \ref{thm:bayesm} for   the case $k =1$.

\begin{corollary}
	\label{cor:lopital}  Assume  the  conditions  of  Theorem \ref{thm:bayesm}.
	Given a point $x_0\in \Xx \setminus S$ assume that
	$\pb_\theta (x_0) = 0$ for all $ \theta \in \Theta$. If the condition for differentiation w.r.t. $r$  at $0$ under 
	the   
	integral $\int_C  \pb_\theta (D_r(x_0))d\mu_\Theta$ holds  for  $C \in \Uu_\Theta \cup \{ \Theta \}$, then  for any $B \in \Uu_\Theta$  we have
	\begin{equation}\label{eq:lim1}
	\mu_{\Theta|\Xx}(B|x_0) =   \frac{ \int _B  \frac{d}{dr}|_{r=0}\pb_\theta (D_r (x_0)) d\mu_\Theta}{\int_\Theta  \frac{d}{dr}|_{r=0}\pb_\theta (D_r (x_0)) d\mu_\Theta},
	\end{equation}
	if the dominator in the RHS of (\ref{eq:lim1})  does not vanish.
\end{corollary}

\begin{remark}\label{rem:classical}   Assume that $\{ \pb_{\theta}|\: \theta \in \Theta\}$ is a  family of measures dominated  by a measure $\nu \in \Pp (\Xx)$, i.e.  for each $\theta \in \Theta$ there exists   $ f_{\Xx|\Theta} (	\cdot|\theta) \in L^1(\Xx, \nu)$ such that $\pb_{\theta} = f_{\Xx|\Theta} (\cdot|\theta) \nu$.  Then  
	\begin{equation}\label{eq:int1}
	\mu(B|x)=\int_B \frac{d\mu_{\Theta|\Xx}}{d\mu_\Theta}(\theta|x) d\mu_\Theta
	\end{equation}
	where, by  Bayes' formula   we have  \cite[Theorem 1.31, p.16]{Schervish1997} 
	\begin{equation}
	\label{eq:schervish}
	\frac{d\mu_{\Theta|\Xx}}{d\mu_\Theta}(\theta|x)= \frac{f_{\Xx| \Theta}(x|\theta)}{\int_{\Theta} f_{\Xx|\Theta}(x|t)d\mu_\Theta  (t)}
	\end{equation}	
	for $\mu_\Xx$-a.e.   $ x\in \Xx$. We  regard    both   (\ref{eq:updatem})  and (\ref{eq:schervish})   as recipes
	for computing     the  posterior    distribution $\mu(B|x)$ under different
	assumptions. 
\end{remark}

\begin{remark}\label{rem:dui}  Assume the conditions of Theorem \ref{thm:bayesm}. 
	Given $x_0 \in \Xx \setminus S$, assuming that
	$\pb_\theta (x_0) = 0$ for all $ \theta \in \Theta$,  a sufficient   condition  for       differentiation under the integral sign  for any $C \in \Uu_\Theta \cup \Theta$
	\begin{equation}\label{eq:lim2}
	\lim_{r \to 0}\frac{1}{r}\int_C \pb_\theta (D_r(x_0))\, d\mu_\Theta = \int_C \frac{d}{dr}|_{r=0}\pb_\theta (D_r (x_0))\, d\mu_\Theta
	\end{equation}
	is  the  differentiability in $r$ of   the function $\pb_\theta (D_r(x_0))$ of the variable $r$  in a neighborhood of $0 \in [0,1]$  and the existence  of a $\mu_\Theta$-measurable function $F(\theta, x_0)$ of the variable $\theta$ such that
	$|\frac{d}{dr}\pb_\theta (D_r (x_0))| \le F(\theta, x_0) $ in this  neighborhood, see e.g. \cite[Theorem 16.11, p. 213]{Jost2005},  which  is also valid for an   arbitrary  measurable  space $\Theta
	$.
\end{remark}

\underline{Step 2} {\it Completion of the proof of Theorem \ref{thm:bayesm}}.   Let $\mu_k$ denote  the joint    distribution  on  $\Theta \times  \Xx^k$  defined   by    the  Bayesian statistical model $(\Theta, \mu_\Theta, \pb ^k, \Xx^k)$.  Given   $B \in \Sigma_\Theta$,  we  denote
by $\nu_1^k$ the marginal probability measure $\mu_{\Xx^k}= (\Pi_{\Xx^k})_* (\mu_k)$  and by
$\nu_2 ^k$  the  measure  $(\Pi _{\Xx^k})_*(1_{B\times \Xx^k}\mu_k)$.
Since  $\nu_1 ^k = (\pb^k)_* \mu_\Theta$ and  $\nu_2 ^k = (\pb^k)_*(1_B\mu_\Theta)$ we  obtain
\begin{equation}\label{eq:power}
\nu_1 ^k = (\nu_1 ^1) ^k \text{ and } \nu_2^k = (\nu_2)^k.
\end{equation}
From   (\ref{eq:power}) it follows  that  the RHS  of (\ref{eq:updatem}) exists   and  equals  to  the  LHS  of  (\ref{eq:updatem}). This completes Theorem \ref{thm:bayesm}.

\

\begin{proposition}\label{prop:cons}
	For any ${\bf x}= (x_1, \cdots,   x_k )\in    (\Xx\setminus  S)^k$ 
	the set $S$  is also   of  zero   marginal  measure $\mu^{{\bf x}}_\Xx$ w.r.t.    the  updated    posterior  measure  $\mu_{\Theta|\Xx}(\cdot |{\bf x})$
\end{proposition}
\begin{proof}
	Let us compute  the    updated  marginal measure  of  $S$
	\begin{eqnarray}
	\mu ^{{\bf x}}_\Xx (S)=  (\pb_*\mu_{\Theta|\Xx}(\cdot |{\bf x})) (S)= \int_\Theta \pb_\theta (S)d	\mu_{\Theta|\Xx}(\cdot |{\bf x})\nonumber\\
	= \lim_{r\to 0} \frac{\int_\Theta \pb_\theta (S)\Pi_{i=1}^k  \pb_\theta(D_r (x_i))d\mu_\Theta}{\int_\Theta \Pi_{i=1}^k \pb_\theta (D_r(x_i))d\mu_\Theta}= 0 \label{eq:upd2}
	\end{eqnarray}
	since   for all $r$ we have
	$$	\int_\Theta \pb_\theta (S)\Pi_{i=1}^k  p_\theta(D_r (x_i))d\mu_\Theta  = 0.$$
	This completes  the proof  of   Proposition \ref{prop:cons}.
\end{proof}

Proposition  \ref{prop:cons}  implies  that     our    formula  for
posterior  distributions  is consistent   for Bayes  inference  see Remark \ref{rem:post}(2)  and \cite[p.6]{GV2017}. We postpone   a  discussion on   posterior   consistency  in   another  paper.

\begin{remark}\label{rem:connected}  We can slightly generalize 	Theorem \ref{thm:bayesm}  to the  case  of  a   finite  dimensional   complete     Riemannian manifold  with a countable  number of connected   components  by 
	setting   the  distance  between   two  points  from different  connected
	components   to be $\infty$.
\end{remark}

\begin{example}\label{ex:bayes}  Assume  that  $\Xx \subset (M, g)$
	is a  closed (or open) subset  of $(M, g)$.  Then  $\Xx$ is a   Polish space
	and $\Pp (\Xx)$  is a Polish space and hence a Souslin space.   Theorem 	\ref{thm:bayesm} implies   that       for  any   (prior)  measure
	$\mu_\Theta  \in \Pp ^2(\Xx)$  our formula (\ref{eq:updatem})  gives   a  posterior    distribution
	of $\mu_\Theta$ on $\Theta: = \Pp(\Xx)$ after  seeing  data 
	${\bf x} \in \Xx ^k$  for any $k \in \N^+$.  If
	$\dim \Xx \ge 1$  then  $\Pp(\Xx)$ is  not  a family of dominated measures  and therefore  we cannot apply     Bayes'  formula  for
	computing  the  posterior  distribution of $\mu_\Theta$  after  seeing  
	data   ${\bf x}$, unless  we have  a   constraint on $\mu_\Theta$,  e.g., $\mu_\Theta$  is a Dirichlet measure, see   \cite[Theorem 4.6, p.62]{GV2017}  for  a short  proof  of computing  the  posterior  distribution  of  Dirichlet measures.
\end{example}

\section{Dirichlet   measures revisited}\label{sec:dirichlet}
In this  section,  first we    discuss  some    functorial methods   of generating   probability measures  on $\Pp(\Xx)$. Then, by using the functorial language  of probabilistic morphisms,
we  revisit Dirichlet distributions,  which   are Dirichlet  measures  on $\Pp(\Xx)$ where  the $\Xx$ are finite  sample spaces, see   Definition \ref{def:dirichlet} and  the  remark  thereafter.  Finally  we    give  a new  proof  of  the existence   of  Dirichlet
measures over      any measurable  space by using  a  functorial  property  of the   Dirichlet map constructed by Sethuraman (Theorem \ref{thm:dir}).

\subsection{Probability measures  on  $\Pp(\Xx)$}\label{subs:p2}
There  are   several known techniques for construction  of  random measures, i.e.,  measures on $\Pp(\Xx)$, see  \cite[Chapter 3]{GV2017}  for an extensive   account.
In this  section    we  shall use  probabilistic   morphisms   for
construction  of  a  probability
measure  on $\Pp (\Xx)$.  The most natural  way  is    to look  at  a  ``simpler" measurable space  $\Xx_s$     and construct   a     probabilistic morphism $T : \Xx_s \leadsto \Xx$  together  with  a   measure  $\mu \in \Pp (\Xx_s)$,  a    measure  $\mu_\Pp \in \Pp^2 (\Xx_s)$,  and   examine if  the measures  $ P_*(\overline  T)( \mu), \, ev_P \circ  P_*^2 (\overline T)(\mu_\Pp),\, P_*^2 (T)(\mu_\Pp)  \in \Pp^2 (\Xx)$     satisfy our requirement. 
$$
\xymatrix{\Xx_s \ar@<1ex>[r]^{\delta}\ar@{~>}[d]^{T}\ar[dr]^{\overline T} &\Pp (\Xx_s) \ar@{~>} [l]^{ev}\ar@<1ex>[r]^{\delta}\ar[d]^{P_*(T)}\ar[dr]^{P_*(\overline T)} & \Pp^2 (\Xx_s) \ar[l]^{ev_P}\ar[d]^{P^2_*(T)}\ar[dr]^{P_*^2 (\overline T)} & \\
	\Xx \ar[r]^{\delta} &\Pp (\Xx)\ar@<1ex>@{~>} [l]^{ev}\ar[r]^{\delta} & \Pp^2 (\Xx) \ar@<1ex>[l]^{ev_P}\ar[r]^{\delta} &  \Pp^3(\Xx) \ar@<1ex>[l]^{ev_P}
} 
$$
By   Lemma  \ref{lem:ev2} (3), $ev_P \circ P_*^2 (\overline T)(\mu_\Pp)= P_*(\overline  T)(ev_P (\mu_\Pp))$, hence   we  need only  to look  at  $P_*(\overline T) (\mu) $ and  $P_* ^2  (T)  (\mu_\Pp)$. Both  these
constructions   are utilized  in the  proof  of  Theorem \ref{thm:dir}, where   we give an alternative  proof of   the Sethuraman  theorem.

Let us recall  that  the  $\sigma$-algebra    on $\Pp(\Xx)$  is generated
by  the      subsets  $\la  A, B^*\ra_P : = \la A, B^* \ra \cap \Pp(\Xx)$  where 
$A \in \Sigma_\Xx$ and  
$B^* \in \Bb ([0,1])$ (see the proof of Prop \ref{prop:sigmainduced}).    Hence  we   obtain the following  easy Lemma  whose proof  is omitted.

\begin{lemma}\label{lem:p2}  Any  $\mu_\Pp \in \Pp^2(\Xx)$  is determined
	uniquely   by  its  value    on the     measurable  subsets
	$\la  A_1,   B_1 ^* \ra_P  \cap  \cdots  \cap \la A_n,  B_n ^* \ra_P $  where  $A_i \in \Sigma _\Xx$, $B_i ^*  \in \Bb ([0,1])$  and 
	$n \in \N$.
\end{lemma}

Note that  the     collection  $\{ (A_1 \times   B_1^*)\times  \cdots  \times (A_n \times B_n^*)\}$  of measurable subsets
generates the $\sigma$-algebra  in  $\Xx_{univ} : = ( (\Xx \times [0,1])^\infty, \Sigma_{\Xx_{univ}}: =  (\Sigma _\Xx \otimes \Bb([0,1]))^\infty)$, moreover   any measure $\nu \in \Pp(\Xx_{univ})$ is determined  uniquely   by its  values  on   the subsets in this  collection.  This suggests   that  we could   take $\Xx_s : = \Xx_{univ}$ to define  a prior  measure  on $\Pp (\Xx)$ 
such that  the obtained  prior measure on $\Pp(\Xx)$ satisfies    Ferguson's required   properties: the  support  of the prior distribution should  be large and the posterior  distributions should be manageable analytically.  
A required  probabilistic morphism $\Xx_{univ} \leadsto \Xx$  has been constructed in   Sethuraman's proof  of the    existence of Dirichlet measures \cite{Sethuraman1994},  which  we  shall   revisit  at the end of this  section.

\subsection{Dirichlet distributions}\label{subs:dir}
Dirichlet distributions are most commonly used as the prior distribution over   finite  sample  spaces  in Bayesian statistics.   In this subsection, following \cite[\S 3.1.1]{GR2003}  and  \cite{Ferguson1973}, we 
recall the notion  of a Dirichlet distribution  $Dir (\alpha_1,\cdots, \alpha_k)$ on $\Delta_k = \Pp(\Om_k)$, where $\alpha:= (\alpha_1, \cdots, \alpha_k)\in \R^k_{\ge 0} \setminus \{ 0\}$  is a parameter of the distribution. Classically,
Dirichlet distributions  $Dir(\alpha)$ are defined    for
$\alpha \in \R^k_{>0}$ 
but  
Ferguson's definition of a Dirichlet  distribution  (Definition \ref{def:dir1}) extends  naturally to a  definition of a   Dirichlet measure  in   general case (Definition  \ref{def:dirichlet})  when $\Xx$ need not be finite,   while the  classical  definition doesn't.chlet  distribution  on $\Om (\alpha)$.

\begin{definition}\label{def:dir1}(cf. \cite[Definition 3.1.1, p. 89]{GR2003}, \cite[p. 211]{Ferguson1973})  Given  $\alpha \in \Mm^*(\Om_k)$ let  $\Om(\alpha) : = \{ \om _j \in  \Om_k |\, \alpha (\om_j )\not = 0\}$. Let $\pi_\alpha: \Om (\alpha)  \to  \Om_k$  denote the natural inclusion.  Let  $l(\alpha): = \#\Om(\alpha)$.
	{\it The Dirichlet  distribution  $Dir(\alpha)\in  \Pp^2 (\Om_k)= \Pp (\Delta_k)$} is the measure  $P_*^2(\pi_{\alpha}) Dir(\alpha|_{\Om (\alpha)})$,  where  $Dir (\alpha|_{\Om(\alpha)})\in \Pp^2(\Om(\alpha))$    is the classical $(l(\alpha)-1)$-dimensional Dirichlet distribution on $\Delta_{l(\alpha)}$, i.e. the distribution  with  the  following density function with respect to the $(l(\alpha)-1)$-dimensional Lebesgue measure  on $\Delta_{l(\alpha)}$:
	\begin{equation}
	\label{eq:dir1}
	Dir(\alpha|_{\Om(\alpha)})(x_{i_1}, \cdots, x_{i_{l(\alpha)}}): = 
	\frac{\Gamma (\alpha_{i_1}  + \cdots    + \alpha_{i_{l(\alpha)}})}{\Gamma (\alpha_{i_1})\cdots  \Gamma(\alpha_{i_{l (\alpha)}})}\Pi_{ j =1} ^{l(\alpha)}  x_{i_j} ^{\alpha_{i_j}-1},
	\end{equation}
	where $(x_{i_1}, \cdots  , x_{i_{l(\alpha)}}) \in \Delta_{l(\alpha)}$.
\end{definition}

We summarize  important known   properties  of Dirichlet  distributions,  which shall be needed later,  in the following

\begin{proposition}\label{prop:dirconv}  
	(1) The map  $Dir_k: (\Mm^*(\Om_k), \tau_w) \to (\Pp^2 (\Om_k), \tau_w):
	\alpha \mapsto  Dir (\alpha)$ is continuous  \cite[11(a), p. 93]{GR2003}.

	(2)    
	If $p \in \Delta_k$ is distributed  by $Dir (\alpha_1, \cdots, \alpha_k)$  then for any partition
	$A_1, \cdots,  A_l$ of $\Om_k$   the   vector\\  $ ( \sum_{\om_i \in A_1}p_i,  \sum_{\om_i \in A_2} p_i , \cdots,  \sum_{\om_ i \in A_l} p_i)$ 
	is distributed   by $Dir(\alpha  (A_1), \cdots , \alpha(A_l) )$ \cite[p. 90]{GR2003}. 
	
	(3)  Let  $\alpha \in \Mm^* (\Om_k)$ and  $(\Pp (\Om_k), Dir(\alpha),  Id_\Pp, \Om_k)$  be a Bayesian   statistical  model. 
	Then the posterior distribution   of the prior  distribution $Dir (\alpha)$ after seeing the data $x \in \Om_k$ is $Dir (\alpha  + \delta_x)$ \cite[p. 212]{Ferguson1973}, \cite[p. 92]{GR2003}.
	
\end{proposition}

\begin{remark}\label{rem:functd} (1)  
	Let  $\Xx = A_1 \dot\cup \cdots \dot\cup A_n$  be a measurable  partition
	of $\Xx$ into $n$ disjoint measurable subsets.  This  partition  induces  a measurable  map $\pi: \Xx \to \Om_n := \{ A_1, \cdots, A_n\}$,  $ \pi(x): = A_i $ if $ x\in A_i$.  
	We have
	$M_*(\pi) (\alpha) = (\alpha(A_1), \cdots, \alpha(A_n))$  for any $\alpha \in \Mm(\Xx)$.  Hence
	the  assertion  (2)  in  Proposition \ref{prop:dirconv}  is   equivalent  to the commutativity of the  following diagram
	for  any $ n \ge k$   and 
	any surjective mapping  $\pi_{nk}: \Om_n \to \Om_k$
	$$
	\xymatrix{\Mm^*(\Om_n) \ar[r]^{Dir_n}\ar[d]^{M_*(\pi_{nk})} &  \Pp^2(\Om_n)\ar[d]^{P^2_*(\pi_{nk})}\\
		\Mm^*(\Om_k) \ar[r]^{Dir_k}&  \Pp^2(\Om_k).
	}$$
	
	(2)  Let  $k < l$  and  $\pi_{kl}: \Om_k \to  \Om_l$ is  an injection.  Let
	$\pi_{lk}: \Om_l \to \Om_k$   be a left inverse   of  $\pi_{kl} $.
	By the definition  of   $Dir (\alpha)$ it  is   not hard  to    see that  the following  diagram is  commutative 
	$$
	\xymatrix{\Mm^*(\Om_k)\ar[r]^{M_*(\pi_{kl})}\ar[d]^{Dir_k}& \Mm^*(\Om_l)\ar[r] ^{ M_*(\pi_{lk})}\ar[d] ^{Dir_l }&  \Mm^*(\Om_k)\ar[d]^{Dir_k}\\
		\Pp^2(\Om_k)\ar[r] ^{P^2_* (\pi_{kl})}&  \Pp^2(\Om_l)\ar[r]^{P_*^2 (\pi_{lk})} &  \Pp^2(\Om_k).
	}
	$$
	It  follows  that  for  any map $\pi_{kl}: \Om_k  \to \Om_l$  we  have
	$$ P_*^2 (\pi_{kl})\circ Dir_k = Dir_l \circ M_*(\pi_{kl}).$$
	Hence $Dir_k : \Mm ^* (\Om_k) \to \Pp^2 (\Om_k)$ is a  natural transformation of    the functor  $M_*$  to the functor  $P_*^2$  in the  category of finite sample  spaces  $\Om_k$   whose morphisms are   (measurable) mappings.

	(3)   The  converse   statement   of     Remark \ref{rem:functd}(2) is also valid, assuming  a  normalization  condition.   Namely 
	for  each $n \in \N^+$ there  exists a unique	   mapping  $Dir_n : \Mm^* (\Om_n) \to  \Pp^2 (\Om_n)$    satisfying  the following  properties.
	
	(i) Normalization:   for  $n = 2$   we have  $Dir_2(\alpha) = Dir (\alpha)$.
	
	(ii) Naturality:  for  any $ n, k\ge 1$  and any  mapping  $\pi_{nk} : \Om_n \to \Om_k$  the following diagram is commutative
	$$
	\xymatrix{\Mm^*(\Om_n) \ar[r]^{Dir_n}\ar[d]^{M_*(\pi_{nk})}&  \Pp^2(\Om_n)\ar[d]^{P^2_*(\pi_{nk})}\\
		\Mm^*(\Om_k) \ar[r]^{Dir_k}&  \Pp^2(\Om_k).
	}$$	
\end{remark}

\subsection{Dirichlet  measures}\label{subs:dirpro}
\begin{definition}\label{def:dirichlet}(cf. \cite[Definition 1, p. 214]{Ferguson1973})  Let $\Xx$ be a measurable space  and $\alpha \in \Mm^*(\Xx)$. An element $\Dd(\alpha) \in \Pp^2(\Xx)$  is called
	a {\it Dirichlet measure
		on $\Pp(\Xx)$ parameterized by $\alpha$}, if  for all  surjective  measurable  mappings $\pi_k : \Xx \to \Om_k$
	we have
	\begin{equation}\label{eq:dirichlet}
	P^2 _* (\pi_k) (\Dd(\alpha))= Dir(\alpha (\pi_k^{-1}(\om_1)), \cdots , \alpha(\pi_k^{-1}(\om_k)))\in \Pp^2(\Om_k) = \Pp(\Delta_k),
	\end{equation}
	where $Dir(\alpha_1, \cdots, \alpha_k)$  
	is the  Dirichlet distribution  with the parameter $(\alpha_1, \cdots, \alpha_k)$ on $\Delta_k$  defined in Definition \ref{def:dir1}. If $\Dd(\alpha)$ is  defined for all $\alpha \in \Mm^*$   we shall call $\Dd:\Mm^* (\Xx) \to \Pp^2 (\Xx)$ {\it a Dirichlet map}.
\end{definition}

Proposition  \ref{prop:dirconv} (2)   implies   that  $Dir(\alpha_1, \cdots, \alpha_k)$ is   a  Dirichlet  measure
on $\Pp(\Om_k)$.

\begin{theorem}\label{thm:dir}
	For any  measurable    space
	$\Xx$ there  exists  a measurable  mapping  $\Dd: \Mm^*(\Xx) \to \Pp^2 (\Xx)$ such that  $\Dd(\alpha)$  is a Dirichlet  measure  parameterized   by $\alpha$.
	Moreover  the  mapping $\Dd$ is  a  natural transformation   of the  functor $M_*$ to the functor $P_*^2$ in the
	category  of   measurable  spaces    whose  morphisms are measurable mappings.	
\end{theorem}

\begin{proof}  We shall show that the    Dirichlet  mapping $\Dd: \Mm^*(\Xx)   \to \Pp^2 (\Xx)$ constructed  by  Sethuraman \cite{Sethuraman1994}   satisfies  the  condition   in Theorem  \ref{thm:dir}(1).	
	Let us recall the construction of $\Dd$.
	First we define  a  mapping $T: \Mm^*(\Xx) \to \Pp( \Xx_{univ})$ as   follows
	\begin{equation*}
	T(\alpha):= (\pi_\Xx (\alpha)\times  beta (1, \alpha (\Xx)))^\infty \in \Pp((\Xx \times [0,1])^\infty).
	\end{equation*}
	Here, $beta$ is the usual beta distribution.
	Since  $\alpha (\Xx) >0$,   the  map  $T_1: \Mm^*(\Xx) \to \Pp^*([0,1]), \alpha \mapsto    beta (1, \alpha(\Xx))$ is a measurable map. Since  $\pi_\Xx$ is    a measurable map by Proposition \ref{prop:tauv},
	the  map  $T$  is measurable.

	Let  $\pi_i : \Xx_{univ} \to \Xx \times [0,1]$ denote   the projection  on the
	$i$-factor  of  $\Xx_{univ}$. Denote by  $j_1$  and  $j_2$  the projections  from $\Xx \times [0,1]$  to  $\Xx$ and to $[0,1]$, respectively.
	The   we define  for  $i \in \N$  and $n \in \N$  the following   measurable mappings:\\
	$\bullet$ $\theta_i: \Xx_{univ} \to \R,\,   \theta_i(x_{univ}): = j_2 \circ \pi_i (x_{univ})$,\\
	$\bullet$ $q_i: \Xx_{univ} \to \Xx, \,  q_i(x_{univ}):=   j_1 \circ \pi_i  (x_{univ})$,\\
	$\bullet$ $p_1:  \Xx_{univ} \to \R,\, p_1 (x_{univ}): =\theta_1 (x_{univ})$,   \\
	$\bullet$  $p_n : \Xx_{univ} \to \R, \,  p_n(x_{univ}): = \theta_n(x_{univ}) \Pi_{i =1}^{n-1}(1-\theta_i(x_{univ}))$,   if  $n \ge 2$,  \\
	$\bullet$ $p_{univ}: \Xx_{univ} \to \Pp (\Xx) , \
	p_{univ}(x_{univ})  : = \sum_{i=1}^\infty  p_i(x_{univ}) \cdot \delta _{q_i(x_{univ})} $\\
	$\bullet$  $\Dd: \Mm^* (\Xx) \to \Pp^2 (\Xx), \,  \Dd (\alpha): = (p_{univ})_* (T_{univ} (\alpha))$.
	
	\begin{remark}\label{rem:puniv}     It is not hard to see   that    for any $x_{univ} \in \Xx_{univ}$  we have  $ \sum_{i=1}^\infty p_i (x_{univ}) \in [ 0,1]$. Therefore
		the  image  of $p_{univ}(\Xx_{univ}) \subset \Mm(\Xx)$. Furthermore it is known  it    that  there  exists  a    measurable subset $ \Xx_{univ}^{reg}$ of   full $\mu_{univ}$-measure  for   any $\mu_{univ} =\otimes ^\infty  (\mu_\Xx \times \nu)\in \Pp(\Xx_{univ})$, where  $\mu_\Xx \in \Pp(\Xx)$ and  $\nu \in \Pp([0,1]) $ is dominated by the Lebesgue measure on $[0,1]$,  such that   the map  $p_{univ}(\Xx^{reg}) \subset  \Pp (\Xx)$, see e.g. \cite[Lemma 3.4, p. 31]{GV2017}.  Thus   the expression   $(p_{univ})_* (T(\alpha))$   should
		be  understood as  $(p_{univ})_* \circ r(\alpha_{univ})	\in \Pp ^2 (\Xx)$ where
		$r:  \Pp (\Xx_{univ}) \to \Pp (\Xx_{univ} ^{reg})$ is   the restriction and $\alpha_{univ} : = T(\alpha)$. 
	\end{remark}
	Since   the mappings $p_i$ are measurable, and  $\sum_{i=1}^\infty p_i (x_{univ}) = 1$ for $x_{univ}\in \Xx_{univ}^{reg}$, the  map  $\Dd(\alpha): \Mm^*(\Xx) \to \Pp^2_*(\Xx)$ is a  measurable  map.  
	
	Now  we     shall show   that       Sethuraman's  Dirichlet  map $\Dd$ is  a natural transformation  of the functor   $M_*$ to the functor  $P_*^2$.

	Let  $\kappa: \Xx \to \Yy$ be a measurable   mapping.  
	Denote by $\kappa_\infty: \Xx_{univ}  \to \Yy_{univ}$ the  induced measurable mapping:
	$$\kappa_\infty ((x_1, \theta_1), \cdots,  (x_\infty, \theta_\infty)): = ((\kappa(x_1), \theta_1), \cdots , (\kappa(x_\infty), \theta_\infty)) . $$
	Clearly  we have $\kappa_\infty (\Xx_{univ}^{reg} ) \subset   \Yy_{univ} ^{reg}$. 
	Let us consider  the 
	following   diagram
	\begin{equation}
	\label{eq:sth1}
	\xymatrix{
		\Mm^* (\Xx)\ar[r]^{r\circ T} \ar[d]^{ \kappa_*} &\Pp(\Xx_{univ}^{reg})\ar[r]^{(p_{univ})_*} \ar[d]^{(\kappa_\infty)_*} & \Pp^2 (\Xx)\ar[d] ^{ P^2_*(\kappa)} \\
		\Mm^*(\Yy)\ar[r]^{r\circ T} & \Pp(\Yy_{univ}^{reg})\ar[r]^{(p_{univ})_*}& \Pp^2(\Yy).
	}
	\end{equation}
	Assume that   $\alpha  \in \Mm^* (\Xx)$. Since  $\kappa_*(\alpha)(\Yy) = \alpha (\Xx)$ we  obtain
	\begin{eqnarray}
	r\circ 	T \circ \kappa_*  (\alpha) =  r \circ ((\pi_\Yy(\kappa_* \alpha) \times \beta  (1, \alpha (\Xx)))^ \infty  \nonumber \\
	= (\kappa _\infty)_* r\circ T(\alpha) \in \Pp  (\Yy_{univ}). \label{eq:funcs1}
	\end{eqnarray}
	Let $\la  A_i, B_i ^*\ra_P \in \Sigma_{\Pp (\Yy)}$. Then
	\begin{equation}\label{eq:sth2}
	(p_{univ})_* (\kappa_\infty)_*  (r(\alpha_{univ})) (\cap_i\la  A_i, B_i ^* \ra_P)= 
	r(	\alpha_{univ}) ( \kappa _\infty ^{-1} (p_{univ} ^{-1} ( \cap_i\la A_i,  B^*_i\ra_P))),
	\end{equation}
	\begin{eqnarray}
	P_*^2 (\kappa) (p_{univ})_* (r(\alpha_{univ}))(\cap_i\la A_i,  B^*_i \ra_P )=  r(\alpha_{univ})(p_{univ} ^{-1} (\kappa_*) ^{-1} (\cap_i \la A_i, B^*_i\ra_P))\nonumber\\ 
	=r(\alpha_{univ})(p_{univ} ^{-1}(\cap_i\la  \kappa ^{-1}(A_i), B^*\ra_P)).\label{eq:sth3}
	\end{eqnarray}	
	Let  $\la A, B^* \ra_P \in \Sigma_{\Pp(\Yy)}$. Then we have
	\begin{equation}\label{eq:sth4}
	p_{univ} ^{-1} (\la A, B^*\ra_P)=\{ y_{univ}\in \Yy_{univ} | \, \sum_{i=1}^\infty p_i (y_{univ}) \delta_{q_i (y_{univ})} (A)\in B^*\}.
	\end{equation}
	Hence
	\begin{eqnarray}
	\kappa _\infty ^{-1}(p_{univ} ^{-1}(\la  A,  B^* \ra_P ))= \{ x_{univ} \in \Xx_{univ}| \,  \sum_{ i =1} ^\infty  p_i  (\kappa_\infty (x_{univ})) \delta _{ q_i (\kappa_\infty (x_{univ}))}\in A\}\nonumber\\
	=\{ x_{univ}\in \Xx_{univ}|\, \sum_{i =1} ^\infty p_i (x_{univ})\delta_{q_i(x_{univ})} \in \kappa ^{-1} (A)\nonumber \\
	= p_{univ}^{-1}( \la\kappa ^{-1} (A), B \ra_P).\label{eq:sth5}
	\end{eqnarray}
	
	It follows  from (\ref{eq:sth5}), (\ref{eq:sth4}) and (\ref{eq:sth3})  that
	\begin{equation}\label{eq:sth6}
	\kappa_\infty ^{-1} ( p _{univ}  ^{-1} ( \cap _i \la   A_i, B_i^* \ra_P))
	= p_{univ}^{-1}  (\kappa_*)^{-1} (\cap _i  \la A_i, B_i ^*\ra_P).
	\end{equation}
	Taking into account Lemma \ref{lem:p2},   
	we obtain  from  (\ref{eq:sth6})  the following  identity
	\begin{equation}\label{eq:sth7}
	(p_{univ})_* (\kappa_\infty)_* ( r(\alpha _{univ}))=  P _*^2  (\kappa) (p_{univ})_*  (r(\alpha _{univ})).
	\end{equation}
	
	Hence we  deduce    from  (\ref{eq:funcs1})  and  (\ref{eq:sth7})  the naturality of the transformation $\Dd$ which  is expressed  in the following
	identity
	\begin{equation}\label{eq:funcs2}
	\Dd \circ M_* (\kappa )(\alpha) = P_* ^2  (\kappa) \circ \Dd (\alpha).
	\end{equation}
	
	As  in Remark \ref{rem:functd},   we    observe    that    the functoriality   of  the  map  $\Dd$ implies  that
	$\Dd(\alpha)$ is a Dirichlet  measure, since it is known that the restriction   of the     map $\Dd$ to the   category  of   finite  sample   spaces    is  the Dirichlet map, see e.g. \cite[\S 3.3.3]{GV2017}. (By Remark \ref{rem:functd}(3)  we need  only to show that   the Sethuraman  map   defined  on $M_*(\Om_2)$  is  the Dirichlet map).
	This completes   the proof   of Theorem \ref{thm:dir}.

\end{proof}

\begin{remark}\label{rem:dir2}
	1.	  Let  $\Theta$ denote
	$\Pp(\Xx)$. Assume that $\Dd: \Mm^*(\Xx) \to \Pp(\Theta)$ is a Dirichlet map.  In \cite[Theorem 1, p. 217]{Ferguson1973} Ferguson proved  that the posterior
	(conditional) distribution $\Dd_{\Theta|\Xx}(\cdot|x)$  with prior $\Dd(\alpha)$ is equal to  $\Dd(\alpha + \delta_x)$.

	2.
	In \cite{Ferguson1973} Ferguson  made use
	of Kolmogorov's consistency theorem.  Since   $\Pp (\R, \Sigma_w)$
	is not  a measurable  subset of $[0,1]^{\Bb (\R)}$, see e.g., \cite[p. 64]{GR2003}, the  Kolmogorov theorem does  not apply directly and  we need a more refined  technique  \cite[Theorem 3.12,  p.28]{GV2017}.
	Ferguson also suggested  a second proof of  the existence  of a  Dirichlet map, which is close  to  Sethuraman's  proof \cite{Sethuraman1994}.

	3.
	We have shown  the benefit  of the functorial language  of probabilistic morphisms in this paper. In particular, by using  the concise  functorial language of probabilistic  morphisms,    we don't need  to use   abstract  generators  of random  variables. 
\end{remark}

\section*{Acknowledgement}  The authors  would like  to   thank Tobias Fritz for discussions on probability monads, Markov category  and many useful comments on an early version of this  paper,  Lorenz Schwachh\"ofer,  Juan Pablo Vigneaux  for beneficial comments on  different  versions  of this paper, and Ulrich Menne for  alerting us  on the relation between \cite{JLT2020} and \cite[2.8.9]{Federer1969}. We appreciate Duc Hoang Luu for   Example \ref{ex:prob1} (3) and  helpful suggestions. We are grateful to the  anonymous  referee  for  critical comments and  valuable suggestions. HVL would like   to thank   Domenico Fiorenza   and XuanLong Nguyen   for helpful comments  on  the   first version    of this paper \cite{Le2019}. She     warmly thanks  JJ, TDT  and Duc Hoang Luu for  hospitality  during  her visit  to MPI MIS  Leipzig in 2019 where a  part  of this
paper  was discussed.

\bibliographystyle{acm}

\begin{thebibliography}{10}

\bibitem{Arnold1998}
{\sc Arnold, L.}
\newblock {\em Random dynamical systems}.
\newblock Springer Monographs in Mathematics. Springer-Verlag, Berlin, 1998.

\bibitem{Aumann1961}
{\sc Aumann, R.~J.}
\newblock Borel structures for function spaces.
\newblock {\em Illinois J. Math. 5}, 4 (1961), 614--630.

\bibitem{AJLS2015}
{\sc Ay, N., Jost, J., L{\^e}, H.~V., and Schwachh{\"o}fer, L.}
\newblock Information geometry and sufficient statistics.
\newblock {\em Probability Theory and Related Fields 162}, 1 (2015), 327--364.

\bibitem{AJLS2017}
{\sc Ay, N., Jost, J., L\^e, H.~V., and Schwachh\"ofer, L.}
\newblock {\em Information geometry}, vol.~64 of {\em Ergebnisse der Mathematik
  und ihrer Grenzgebiete. 3. Folge. A Series of Modern Surveys in Mathematics
  [Results in Mathematics and Related Areas. 3rd Series. A Series of Modern
  Surveys in Mathematics]}.
\newblock Springer, Cham, 2017.

\bibitem{AJLS2018}
{\sc Ay, N., Jost, J., Lê, H.~V., and Schwachhöfer, L.}
\newblock Parametrized measure models.
\newblock {\em Bernoulli 24}, 3 (2018), 1692--1725.

\bibitem{Baudoin2014}
{\sc Baudoin, F.}
\newblock {\em Diffusion processes and stochastic calculus}.
\newblock EMS Textbooks in Mathematics. European Mathematical Society (EMS),
  Z\"{u}rich, 2014.

\bibitem{Berger1993}
{\sc Berger, J.~O.}
\newblock {\em Statistical decision theory and {B}ayesian analysis}.
\newblock Springer Series in Statistics. Springer-Verlag, New York, 1993.
\newblock Corrected reprint of the second (1985) edition.

\bibitem{Blackwell1953}
{\sc Blackwell, D.}
\newblock Equivalent comparisons of experiments.
\newblock {\em Ann. Math. Statist. 24}, 2 (1953), 265--272.

\bibitem{Bogachev2007}
{\sc Bogachev, V.~I.}
\newblock {\em Measure theory. {V}ol. {I}, {II}}.
\newblock Springer-Verlag, Berlin, 2007.

\bibitem{Bogachev2010}
{\sc Bogachev, V.~I.}
\newblock {\em Differentiable measures and the {M}alliavin calculus}, vol.~164
  of {\em Mathematical Surveys and Monographs}.
\newblock American Mathematical Society, Providence, RI, 2010.

\bibitem{Bogachev2018}
{\sc Bogachev, V.~I.}
\newblock {\em Weak convergence of measures}, vol.~234 of {\em Mathematical
  Surveys and Monographs}.
\newblock American Mathematical Society, Providence, RI, 2018.

\bibitem{LeCam1964}
{\sc Cam, L.~L.}
\newblock Sufficiency and approximate sufficiency.
\newblock {\em Ann. Math. Statist. 35}, 4 (1964), 1419--1455.

\bibitem{CP1997}
{\sc Chang, J.~T., and Pollard, D.}
\newblock Conditioning as disintegration.
\newblock {\em Statistica Neerlandica 51}, 3 (1997), 287--317.

\bibitem{Chentsov1965}
{\sc Chentsov, N.}
\newblock Categories of mathematical statistics.
\newblock {\em Doklady of Acad.U.S.S.R. 165}, 165 (1965), 511--514.

\bibitem{Chentsov1972}
{\sc Chentsov, N.~N.}
\newblock {\em Statistical decision rules and optimal inference}.
\newblock Izdat. ``Nauka'', Moscow, 1972.

\bibitem{CJ2019}
{\sc Cho, K., and Jacobs, B.}
\newblock Disintegration and Bayesian inversion via string diagrams.
\newblock {\em Mathematical Structures in Computer Science 29}, 7 (2019),
  938–971.

\bibitem{DM1978} {\sc Dellacherie C.  and Meyer  P.-A.}
\newblock {\em Probabilities and potential}.
\newblock North Holland, 1978.

\bibitem{Federer1969} {\sc Federer H.}  
\newblock {\em Geometric measure theory}.
\newblock Die Grundlehren der mathematischen Wissenschaften, Band. 153,  Springer, 1969.


\bibitem{Ferguson1973}
{\sc Ferguson, T.~S.}
\newblock A Bayesian analysis of some nonparametric problems.
\newblock {\em Ann. Statist. 1}, 2 (1973), 209--230.

\bibitem{Fritz2020}
{\sc Fritz, T.}
\newblock A synthetic approach to markov kernels, conditional independence and
  theorems on sufficient statistics.
\newblock {\em Advances in Mathematics 370\/} (2020), 107239.

\bibitem{FP2018}
{\sc Fritz, T., and Perrone, P.}
\newblock Bimonoidal structure of probability monads.
\newblock {\em Electronic Notes in Theoretical Computer Science 341\/} (2018),
  121--149.
\newblock Proceedings of the Thirty-Fourth Conference on the Mathematical
  Foundations of Programming Semantics (MFPS XXXIV).

\bibitem{GH1989}
{\sc Gaudard, M., and Hadwin, D.}
\newblock Sigma-algebras on spaces of probability measures.
\newblock {\em Scandinavian Journal of Statistics 16}, 2 (1989), 169--175.

\bibitem{GV2017}
{\sc Ghosal, S., and van~der Vaart, A.}
\newblock {\em Fundamentals of nonparametric {B}ayesian inference}, vol.~44 of
  {\em Cambridge Series in Statistical and Probabilistic Mathematics}.
\newblock Cambridge University Press, Cambridge, 2017.

\bibitem{GR2003}
{\sc Ghosh, J.~K., and Ramamoorthi, R.~V.}
\newblock {\em Bayesian nonparametrics}.
\newblock Springer Series in Statistics. Springer-Verlag, New York, 2003.

\bibitem{Giry1982}
{\sc Giry, M.}
\newblock A categorical approach to probability theory.
\newblock In {\em Categorical aspects of topology and analysis ({O}ttawa,
  {O}nt., 1980)}, vol.~915 of {\em Lecture Notes in Math.} Springer, Berlin-New
  York, 1982, pp.~68--85.

\bibitem{Golubtsov2002}
{\sc Golubtsov}.
\newblock Monoidal kleisli category as a background for information
  transformers theory (in russian).
\newblock {\em Information Theory and Information Processing 2}, 1 (2002),
  62--84.
\newblock Translated from Russian. jip.ru/2002/GOLU1.pdf.

\bibitem{Jost2005}
{\sc Jost, J.}
\newblock {\em Postmodern analysis}, third~ed.
\newblock Universitext. Springer-Verlag, Berlin, 2005.

\bibitem{JLT2020}
{\sc Jost, J., Lê, H.~V., and Tran, T.~D.}
\newblock {Differentiation of measures on complete Riemannian manifolds}.
\newblock {\em arXiv:2008.13252\/} (2020).

\bibitem{JMPR2019}
{\sc Jost, J., Matveev, R., Portegies, J., and Rodrigues, C.}
\newblock On the regular representation of measures.
\newblock {\em Communications in Analysis and Geometry 27}, 8 (2019),
  1799--1823.

\bibitem{Kallenberg2017}
{\sc Kallenberg, O.}
\newblock {\em Random measures, theory and applications}, vol.~77 of {\em
  Probability Theory and Stochastic Modelling}.
\newblock Springer, Cham, 2017.

\bibitem{Kifer1986}
{\sc Kifer, Y.}
\newblock {\em Ergodic theory of random transformations}, vol.~10 of {\em
  Progress in Probability and Statistics}.
\newblock Birkh\"{a}user Boston, Inc., Boston, MA, 1986.

\bibitem{Kifer1988}
{\sc Kifer, Y.}
\newblock {\em Random perturbations of dynamical systems}, vol.~16 of {\em
  Progress in Probability and Statistics}.
\newblock Birkh\"{a}user Boston, Inc., Boston, MA, 1988.

\bibitem{Lawvere1962}
{\sc Lawvere, W.~F.}
\newblock The category of probabilistic mappings.
\newblock Available at
  https://ncatlab.org/nlab/files/lawvereprobability1962.pdf., 1962.

\bibitem{Le2019}
{\sc L{\^e}, H.~V.}
\newblock 
Bayesian statistical models and  Bayesian machine learning,
\newblock {\em preprint,} (2019).
\bibitem{Le2020}
{\sc L{\^e}, H.~V.}
\newblock {Diffeological Statistical Models, the Fisher Metric and
  Probabilistic Mappings}.
\newblock {\em Mathematics 8}, 2 (2020), 167.

\bibitem{MacLane1994}
{\sc Mac~Lane, S.}
\newblock {\em Categories for the working mathematician}, 6th corrected
  printing~ed., vol.~5 of {\em Graduate Texts in Mathematics}.
\newblock Springer-Verlag, New York, 1994.

\bibitem{MS1966}
{\sc Morse, N., and Sacksteder, R.}
\newblock Statistical isomorphism.
\newblock {\em Ann. Math. Statist. 37}, 1 (1966), 203--214.

\bibitem{Neyman1935}
{\sc Neyman, J.}
\newblock Sur un teorema concernente le considette statistische sufficienti.
\newblock {\em G. Ist. Ital. Attuari 6\/} (1935), 320--334.

\bibitem{Panangaden1999}
{\sc Panangaden, P.}
\newblock The category of markov kernels.
\newblock {\em Electronic Notes in Theoretical Computer Science 22\/} (1999),
  171--187.
\newblock PROBMIV'98, First International Workshop on Probabilistic Methods in
  Verification.

\bibitem{Panangaden2009}
{\sc Panangaden, P.}
\newblock {\em Labelled {M}arkov processes}.
\newblock Imperial College Press, London, 2009.

\bibitem{Parthasarathy1967}
{\sc Parthasarathy, K.~R.}
\newblock {\em Probability measures on metric spaces}.
\newblock Probability and Mathematical Statistics, No. 3. Academic Press, Inc.,
  New York-London, 1967.

\bibitem{Parzygnat2020}
{\sc Parzygnat, A.~J.}
\newblock Inverses, disintegrations, and bayesian inversion in quantum markov
  categories.
\newblock {\em arXiv:2001.08375\/} (2020).

\bibitem{Pfanzagl2017}
{\sc Pfanzagl, J.}
\newblock {\em Mathematical statistics}.
\newblock Springer Series in Statistics. Springer-Verlag, Berlin, 2017.
\newblock Essays on history and methodology, Springer Series in Statistics.
  Perspectives in Statistics.

\bibitem{Sacksteder1967}
{\sc Sacksteder, R.}
\newblock A note on statistical equivalence.
\newblock {\em The Annals of Mathematical Statistics 38}, 3 (1967), 787--794.

\bibitem{Schervish1997}
{\sc Schervish, M.~J.}
\newblock {\em Theory of statistics}.
\newblock Springer Series in Statistics. Springer-Verlag, New York, 1995.

\bibitem{Sethuraman1994}
{\sc Sethuraman, J.}
\newblock A constructive definition of dirichlet priors.
\newblock {\em Statistica Sinica 4}, 2 (1994), 639--650.

\end{thebibliography}

\end{document}